\documentclass[amstex,12pt,russian,amssymb]{article}

\usepackage{mathtext}
\usepackage[cp1251]{inputenc}
\usepackage[T2A]{fontenc}
\usepackage[russian]{babel}
\usepackage[dvips]{graphicx}
\usepackage{amsmath}
\usepackage{amssymb}
\usepackage{amsxtra}
\usepackage{latexsym}
\usepackage{ifthen}

\begin{document}

\newcounter{lemma}
\newcommand{\lemma}{\par \refstepcounter{lemma}%
{\bf Лемма \arabic{lemma}.}}

\newcounter{corollary}
\newcommand{\corollary}{\par \refstepcounter{corollary}%
{\bf Следствие \arabic{corollary}.}}

\newcounter{remark}
\newcommand{\remark}{\par \refstepcounter{remark}%
{\bf Замечание \arabic{remark}.}}

\newcounter{theorem}
\newcommand{\theorem}{\par \refstepcounter{theorem}%
{\bf Теорема \arabic{theorem}.}}

\newcounter{proposition}
\newcommand{\proposition}{\par \refstepcounter{proposition}%
{\bf Предложение \arabic{proposition}.}}

\newcommand{\proof}{{\it Доказательство.\,\,}}
\renewcommand{\refname}{\centerline{\bf Список литературы}}

{\bf Е.А.~Севостьянов} (Житомирский государственный университет им.\
И.~Франко)

\medskip
{\bf Є.О.~Севостьянов} (Житомирський державний університет ім.\
І.~Франко)

\medskip
{\bf E.A.~Sevost'yanov} (Zhytomyr Ivan Franko State University)

\medskip
{\bf О граничном продолжении отображений в терминах простых концов}

{\bf Про продовження відображень на межу в термінах простих кінців }

{\bf On boundary behavior of mappings in terms of prime ends}

\medskip
Изучается граничное поведение классов отображений, относящихся к
изучению классов Соболева и Орлича--Соболева в евклидовом $n$-мерном
пространстве. В терминах простых концов регулярных областей получены
теоремы о непрерывном продолжении указанных классов на границу
области. Более того, изучено глобальное поведение семейств указанных
отображений, в частности, доказаны результаты о равностепенной
непрерывности их семейств в замыкании области в терминах простых
концов.

\medskip
Вивчається межова поведінка класів відображень, стисло пов'язаних з
класами Соболєва і Орліча--Соболєва в евклідовому $n$-вимірному
просторі. В термінах простих кінців регулярних областей отримано
теореми про неперервне продовження зазначених класів на межу
області. Більше того, вивчено глобальну поведінку сімей вказаних
відображень, зокрема, доведено результати про одностайну
неперервність їх сімей в замиканні області в термінах простих
кінців.

\medskip
A boundary behavior of mappings, which are closely related with
Sobolev and Orlicz--Sobolev classes in $n$-measured Euclidean space,
is investigated. There are obtained theorems on continuous boundary
extension of classes mentioned above. Moreover, a global behavior of
the mappings mentioned above was studied. In particular, there is
proved equicontinuity of the classes mentioned above in the closure
of the domain in terms of prime ends.

\newpage

{\bf 1. Введение.} Настоящая работа посвящена изучению отображений с
ограниченным и конечным искажением, активно изучаемых в последнее
время в ряде работ отечественных и зарубежных авторов, см., напр.,
\cite{AC}, \cite{BGMV}, \cite{Cr$_1$}--\cite{Cr$_2$},
\cite{Gol$_1$}--\cite{Gol$_2$},
\cite{GRSY}, \cite{IM}, 
\cite{MRSY}, \cite{MRV$_1$}--\cite{MRV$_3$}, \cite{Re}, \cite{Ri} и
\cite{Vu}. Отдельного внимания заслуживают работы, в которых
изложены результаты, относящиеся к изучению классов Орлича-Соболева
в окрестности границы заданной области в терминах простых концов
(см. \cite{KR} и \cite{GRY}). Здесь же упомянем публикации, в
которых исследовано граничное поведение этих классов в случае
локально связных границ (см., напр., \cite{MRSY}, \cite{KRSS},
\cite{KSS}, \cite{Sev$_3$}).

\medskip
Остановимся теперь на работе \cite{KR}, где получены важные
результаты о граничном поведении гомеоморфизмов, удовлетворяющих
определённым геометрическим ограничениям и тесно связанных с
классами Соболева и Орлича--Соболева. Здесь, в частности, показано,
что один класс гомеоморфизмов может быть продолжен на границу
непрерывным образом, при этом, непрерывность должна пониматься в
терминах так называемых простых концов, поскольку речь идёт о
границах областей с <<плохими>> свойствами.

\medskip
Одной из основных целей настоящей работы является усиление
результатов из упомянутой публикации \cite{KR}. Задача, которую мы
ставим перед собой, состоит в том, чтобы изложить сформированную
теорию граничного продолжения отображений в максимально законченном
виде. Мы достигаем этой цели, распространяя указанные результаты на
отображения с ветвлением и используя здесь в качестве априорных
условий более общие ограничения на коэффициент искажения
отображений. Точнее говоря, прибегая к ограничениям на рост так
называемой дилатации порядка $\alpha$, мы тем самым охватываем
результаты работы \cite{KR}, где схожие ограничения касаются лишь
специально выбранного случая <<дилатации порядка $n$>>. Таким
образом, улучшение утверждений, относящихся к работе \cite{KR},
происходит как бы в двух независимых направлениях: с одной стороны,
мы ослабляем топологические условия на отображения, допуская наличие
ветвлений, с другой стороны, более общая интерпретация коэффициента
искажения также позволяет значительно расширить рассмотренную в
\cite{KR} ситуацию.

Отметим, что настоящая работа состоит как бы из двух частей, тесно
связанных между собой по смыслу и содержанию, но несколько
отличающихся по цели исследования. В первой части речь идёт
преимущественно о вопросах граничного поведения отображений, где
кроме классов Орлича--Соболева изучена возможность непрерывного
граничного продолжения так называемых кольцевых $Q$-отображений и
нижних $Q$-отображений (см. \cite{MRSY}). Эти отображения на данный
момент являются главным инструментом исследования отображений с
конечным искажением (см., напр., \cite{AC}--\cite{GRY}). Изучение
граничного поведения нижних $Q$-отображений приводит к основным
результатам настоящей статьи в качестве следствий. Эта часть работы
заключена в разделах 2--4. Вторая часть содержит результаты о
глобальном поведении отображений в области, точнее, о нормальных и
равностепенно-непрерывных семействах отображений в замыкании области
(где замыкание также следует понимать в смысле простых концов). Для
простоты и большей ясности изложения мы ограничиваемся здесь случаем
гомеоморфизмов. К этой части работы относятся разделы 5 и 6.
Очевидная связь обеих частей состоит в том, что глобальное поведение
отображений может быть изучено лишь после того, как установлено их
непрерывное продолжение на границу области.

\medskip
Приведём теперь необходимые для изложения сведения. Следующие
определения могут быть найдены в работе \cite{KR}. Пусть $\omega$ --
открытое множество в ${\Bbb R}^k$, $k=1,\ldots,n-1$. Непрерывное
отображение $\sigma:\omega\rightarrow{\Bbb R}^n$ называется {\it
$k$-мерной поверхностью} в ${\Bbb R}^n$. {\it Поверхностью} будет
называться произвольная $(n-1)$-мерная поверхность $\sigma$ в ${\Bbb
R}^n.$ Поверхность $\sigma:\omega\rightarrow D$ называется {\it
жордановой поверхностью} в $D$, если $\sigma(z_1)\ne\sigma(z_2)$ при
$z_1\ne z_2$. Далее мы иногда будем использовать $\sigma$ для
обозначения всего образа $\sigma(\omega)\subset {\Bbb R}^n$ при
отображении $\sigma$, $\overline{\sigma}$ вместо
$\overline{\sigma(\omega)}$ в ${\Bbb R}^n$ и $\partial\sigma$ вместо
$\overline{\sigma(\omega)}\setminus\sigma(\omega)$. Жорданова
поверхность $\sigma$ в $D$ называется {\it разрезом} области $D$,
если $\sigma$ разделяет $D$, т.\,е. $D\setminus \sigma$ имеет больше
одной компоненты, $\partial\sigma\cap D=\varnothing$ и
$\partial\sigma\cap\partial D\ne\varnothing$.

Последовательность $\sigma_1,\sigma_2,\ldots,\sigma_m,\ldots$
разрезов области $D$ называется {\it цепью}, если:

\medskip

(i) $\overline{\sigma_i}\cap\overline{\sigma_j}=\varnothing$ для
всех $i\ne j$, $i,j= 1,2,\ldots$;

\medskip

(ii) $\sigma_{m-1}$ и $\sigma_{m+1}$ содержатся в различных
компонентах $D\setminus \sigma_m$ для всех $m>1$;

\medskip

(iii) $\cap\,d_m=\varnothing$, где $d_m$ -- компонента $D\setminus
\sigma_m$, содержащая $\sigma_{m+1}$.

\medskip
Согласно определению, цепь разрезов $\{\sigma_m\}$ определяет цепь
областей $d_m\subset D$, таких, что $\partial\,d_m\cap
D\subset\sigma_m$ и $d_1\supset d_2\supset\ldots\supset
d_m\supset\ldots$. Две цепи разрезов $\{\sigma_m\}$ и
$\{\sigma_k^{\,\prime}\}$ называются {\it эквивалентными}, если для
каждого $m=1,2,\ldots$ область $d_m$ содержит все области
$d_k^{\,\prime}$ за исключением конечного числа и для каждого
$k=1,2,\ldots$ область $d_k^{\,\prime}$ также содержит все области
$d_m$ за исключением конечного числа. {\it Конец} области $D$ -- это
класс эквивалентных цепей разрезов $D$.

Пусть $K$ -- конец области $D$ в ${\Bbb R}^n$, $\{\sigma_m\}$ и
$\{\sigma_m^{\,\prime}\}$ -- две цепи в $K$, $d_m$ и
$d_m^{\,\prime}$ -- области, соответствующие $\sigma_m$ и
$\sigma_m^{\,\prime}$. Тогда
$$\bigcap\limits_{m=1}\limits^{\infty}\overline{d_m}\subset
\bigcap\limits_{m=1}\limits^{\infty}\overline{d_m^{\,\prime}}\subset
\bigcap\limits_{m=1}\limits^{\infty}\overline{d_m}\ ,$$ и, таким
образом,
$$\bigcap\limits_{m=1}\limits^{\infty}\overline{d_m}=
\bigcap\limits_{m=1}\limits^{\infty}\overline{d_m^{\,\prime}}\ ,$$
т.\,е. множество
$$I(K)=\bigcap\limits_{m=1}\limits^{\infty}\overline{d_m}$$ зависит
только от $K$ и не зависит от выбора цепи разрезов $\{\sigma_m\}$.
Множество $I(K)$ называется {\it телом конца} $K$.

Число прообразов
$N(y, S)={\rm card}\,S^{-1}(y)={\rm card}\,\{x\in\omega:S(x)=y\},\
y\in{\Bbb R}^n$ будем называть {\it функцией кратности} поверхности
$S.$ Другими словами, $N(y, S)$ -- кратность накрытия точки $y$
поверхностью $S.$ Пусть $\rho:{\Bbb R}^n\rightarrow\overline{{\Bbb
R}^+}$ -- борелевская функция, в таком случае интеграл от функции
$\rho$ по поверхности $S$ определяется равенством:  $\int\limits_S
\rho\,d{\mathcal{A}}:=\int\limits_{{\Bbb R}^n}\rho(y)\,N(y,
S)\,d{\mathcal H}^ky.$
Пусть $\Gamma$ -- семейство $k$-мерных поверхностей $S.$ Борелевскую
функцию $\rho:{\Bbb R}^n\rightarrow\overline{{\Bbb R}^+}$ будем
называть {\it допустимой} для семейства $\Gamma,$ сокр. $\rho\in{\rm
adm}\,\Gamma,$ если
\begin{equation}\label{eq8.2.6}\int\limits_S\rho^k\,d{\mathcal{A}}\geqslant 1\end{equation}
для каждой поверхности $S\in\Gamma.$ Пусть $p\geqslant 1,$ тогда
{\it $p$-модулем} семейства $\Gamma$ назовём величину
$$M_p(\Gamma)=\inf\limits_{\rho\in{\rm adm}\,\Gamma}
\int\limits_{{\Bbb R}^n}\rho^p(x)\,dm(x)\,.$$ Полагаем также
$M(\Gamma):=M_n(\Gamma).$ Далее, как обычно, для множеств $A$, $B$ и
$C$ в ${\Bbb R}^n$, $\Gamma(A,B,C)$ обозначает семейство всех
кривых, соединяющих $A$ и $B$ в $C$.

Следуя \cite{Na}, будем говорить, что конец $K$ является {\it
простым концом}, если $K$ содержит цепь разрезов $\{\sigma_m\}$,
такую, что
%
$$M(\Gamma(C, \sigma_m, D))=0$$
%
для некоторого континуума $C$ в $D$, где $M$ -- модуль семейства
$\Gamma(C, \sigma_m, D).$

\medskip
Будем говорить, что граница области $D$ в ${\Bbb R}^n$ является {\it
локально квазиконформной}, если каждая точка $x_0\in\partial D$
имеет окрестность $U$, которая может быть отображена квазиконформным
отображением $\varphi$ на единичный шар ${\Bbb B}^n\subset{\Bbb
R}^n$ так, что $\varphi(\partial D\cap U)$ является пересечением
${\Bbb B}^n$ с координатной гиперплоскостью. Говорим, что
ограниченная область $D$ в ${\Bbb R}^n$ {\it регулярна}, если $D$
может быть квазиконформно отображена на область с локально
квазиконформной границей.

\medskip
Как следует из теоремы 4.1 в \cite{Na}, при квазиконформных
отображениях $g$ области $D_0$ с локально квазиконформной границей
на область $D$ в ${\Bbb R}^n$, $n\geqslant2$, существует
естественное взаимно однозначное соответствие между точками
$\partial D_0$ и простыми концами области $D$ и, кроме того,
предельные множества $C(g,b)$, $b\in\partial D_0$, совпадают с телом
$I(P)$ соответствующих простых концов $P$ в $D$.

Если $\overline{D}_P$ является пополнением регулярной области $D$ ее
простыми концами и $g_0$ является квазиконформным отображением
области $D_0$ с локально квазиконформной границей на $D$, то оно
естественным образом определяет в $\overline{D}_P$ метрику
$\rho_0(p_1,p_2)=\left|{\widetilde
{g_0}}^{-1}(p_1)-{\widetilde{g_0}}^{-1}(p_2)\right|$, где
${\widetilde {g_0}}$ продолжение $g_0$ в $\overline{D_0}$,
упомянутое выше.

Если $g_*$ является другим квазиконформным отображением некоторой
области $D_*$ с локально квазиконформной границей на область $D$, то
соответствующая метрика
$\rho_*(p_1,p_2)=\left|{\widetilde{g_*}}^{-1}(p_1)-{\widetilde{g_*}}^{-1}(p_2)\right|$
порождает ту же самую сходимость и, следовательно, ту же самую
топологию в $\overline{D}_P$ как и метрика $\rho_0$, поскольку
$g_0\circ g_*^{-1}$ является квазиконформным отображением между
областями $D_*$ и $D_0$, которое по теореме 4.1 из \cite{Na}
продолжается до гомеоморфизма между $\overline{D_*}$ и
$\overline{D_0}$.

В дальнейшем, будем называть данную топологию в пространстве
$\overline{D}_P$ {\it топологией простых концов} и понимать
непрерывность отображений
$F:\overline{D}_P\rightarrow\overline{D^{\,\prime}}_P$ как раз
относительно этой топологии.

Пусть $\varphi:[0,\infty)\rightarrow[0,\infty)$ -- неубывающая
функция, $f$ -- локально интегрируемая вектор-функция $n$
вещественных переменных $x_1,\ldots,x_n,$ $f=(f_1,\ldots,f_n),$
$f_i\in W_{loc}^{1,1},$ $i=1,\ldots,n.$ Будем говорить, что
$f:D\rightarrow {\Bbb R}^n$ принадлежит классу
$W^{1,\varphi}_{loc},$ пишем $f\in W^{1,\varphi}_{loc},$ если
$\int\limits_{G}\varphi\left(|\nabla f(x)|\right)\,dm(x)<\infty$ для
любой компактной подобласти $G\subset D,$ где $|\nabla
f(x)|=\sqrt{\sum\limits_{i=1}^n\sum\limits_{j=1}^n\left(\frac{\partial
f_i}{\partial x_j}\right)^2}.$ Класс $W^{1,\varphi}_{loc}$
называется классом {\it Орлича--Соболева}. Отображение
$f:D\rightarrow {\Bbb R}^n$ называется {\it дискретным}, если
прообраз $f^{-1}\left(y\right)$ каждой точки $y\in{\Bbb R}^n$
состоит только из изолированных точек. Отображение $f:D\rightarrow
{\Bbb R}^n$ называется {\it открытым}, если образ любого открытого
множества $U\subset D$ является открытым множеством в ${\Bbb R}^n.$
Отображение $f:D\rightarrow {\Bbb R}^n$ называется {\it сохраняющим
границу отображением} (см. \cite[разд. 3, гл. II]{Vu$_1$}), если
выполнено соотношение $C(f,
\partial D)\subset \partial f(D).$ Отметим, что условие сохранения границы для
открытых дискретных отображений эквивалентно тому, что отображение
$f$ {\it замкнуто} (т.е., $f(A)$ замкнуто в $f(D)$ для любого
замкнутого $A\subset D$), а также тому, что $f^{\,-1}(K)$ компактно
в $D$ для любого компакта $K\subset f(D)$ (см.
\cite[теорема~3.3]{Vu$_1$}).

Будем говорить, что граница $\partial D$ области $D$ {\it сильно
достижима в точке $x_0\in
\partial D$ относительно $p$-модуля}, если для любой окрестности $U$ точки $x_0$ найдется
компакт $E\subset D,$ окрестность $V\subset U$ точки $x_0$ и число
$\delta
>0$ такие, что
$$
M_p(\Gamma(E,F, D))\geqslant \delta
$$
для любого континуума  $F$ в $D,$ пересекающего $\partial U$ и
$\partial V.$ (Здесь $M_p$ обозначает модуль семейств кривых, а
$\Gamma(E,F, D)$ обозначает семейство всех кривых, соединяющих
множества $E$ и $F$ в области $D,$ см., напр., \cite[разделы~2.2 и
2.5]{MRSY}). Граница области $D\subset{{\Bbb R}^n}$ называется {\it
сильно достижимой относительно $p$-модуля}, если указанное выше
свойство выполнено в каждой точке $x_0\in\partial D.$

\medskip
Для отображений класса $W_{loc}^{1,1},$ произвольного $p\geqslant 1$
и почти всех $x\in D$ определим следующие величины:
$l\left(f^{\,\prime}(x)\right):=\min\limits_{|h|=1}
{|f^{\,\prime}(x)h|},$ $J(x,f):={\rm det\,}f^{\,\prime}(x),$
\begin{equation}\label{eq0.1.1A}
K_{I, p}(x,f)\quad =\quad\left\{
\begin{array}{rr}
\frac{|J(x,f)|}{{l\left(f^{\,\prime}(x)\right)}^p}, & J(x,f)\ne 0,\\
1,  &  f^{\,\prime}(x)=0, \\
\infty, & \text{в\,\,остальных\,\,случаях}
\end{array}
\right.\,.
\end{equation}
Величина  $K_{I, p}(x,f)$ называется {\it внутренней дилатацией
отображения $f$ порядка $p$ в точке $x$.} Всюду ниже мы полагаем
$K_I(x, f):=K_{I, n}(x,f).$ Будем говорить, что локально
интегрируемая функция $\varphi:D\rightarrow{\Bbb R}$ имеет {\it
конечное среднее колебание} в точке $x_0\in D$, пишем $\varphi\in
FMO(x_0),$ если
%
%
%
%
$$\limsup\limits_{\varepsilon\rightarrow
0}\frac{1}{\Omega_n\varepsilon^n}\int\limits_{B( x_0,\,\varepsilon)}
|{\varphi}(x)-\overline{{\varphi}}_{\varepsilon}|\,
dm(x)<\infty\,,$$
%
%
где
$\overline{{\varphi}}_{\varepsilon}=\frac{1}
{\Omega_n\varepsilon^n}\int\limits_{B(x_0,\,\varepsilon)}
{\varphi}(x)\, dm(x).$
\medskip
Заметим, что $\Omega_n\varepsilon^n=m(B(x_0, \varepsilon)).$
Основным результатом настоящей статьи, относящимся к непрерывному
продолжению классов Орлича--Соболева на границу, является следующая
теорема.

\medskip
Всюду в статье, если не оговорено противное, $Q:{\Bbb
R}^n\rightarrow [0, \infty]$ -- измеримая по Лебегу функция, равная
нулю вне заданной области $D,$ при этом, мы требуем, чтобы
$0<Q(x)<\infty$ при всех $x\in D.$

\medskip
\begin{theorem}\label{th3}
{\,\sl Пусть $n\geqslant 2,$ $\alpha>1,$ область $D\subset {\Bbb
R}^n$ регулярна, а $D^{\,\prime}\subset {\Bbb R}^n$ ограничена и
имеет локально квазиконформную границу, являющуюся сильно достижимой
относительно $\alpha$-модуля. Пусть также отображение
$f:D\rightarrow D^{\,\prime},$ $D^{\,\prime}=f(D),$ принадлежащее
классу $W_{loc}^{1, \varphi}(D)$ является открытым, дискретным и
замкнутым. Тогда $f$ имеет непрерывное продолжение до непрерывного
отображения $f:\overline{D}_P\rightarrow \overline{D^{\,\prime}}_P,$
$f(\overline{D}_P)=\overline{D^{\,\prime}}_P,$ если выполнено
условие
\begin{equation}\label{eqOS3.0a}
\int\limits_{1}^{\infty}\left(\frac{t}{\varphi(t)}\right)^
{\frac{1}{n-2}}dt<\infty
\end{equation}
и, кроме того, найдётся измеримая по Лебегу функция $Q,$ такая что
$K_{I,\alpha} (x, f)\leqslant Q(x)$ при почти всех $x\in D,$ и
выполнено одно из следующих условий:

1) либо в каждой точке $x_0\in
\partial D$ при некотором $\varepsilon_0>0$ и всех $\varepsilon\in
(0, \varepsilon_0)$ выполнены следующие условия:
\begin{equation*}
\int\limits_{\varepsilon}^{\varepsilon_0}
\frac{dt}{t^{\frac{n-1}{\alpha-1}}q_{x_0}^{\,\frac{1}{\alpha-1}}(t)}<\infty\,,\qquad
\int\limits_{0}^{\varepsilon_0}
\frac{dt}{t^{\frac{n-1}{\alpha-1}}q_{x_0}^{\,\frac{1}{\alpha-1}}(t)}=\infty\,;
\end{equation*}

2) либо $Q\in FMO(x_0)$ в каждой точке $x_0\in \partial D.$
}
\end{theorem}
Здесь
$$q_{x_0}(r):=\frac{1}{\omega_{n-1}r^{n-1}}\int\limits_{|x-x_0|=r}Q(x)\,d{\mathcal
H}^{n-1}$$ обозначает среднее интегральное значение функции $Q$ над
сферой $S(x_0, r).$ В частности, заключение теоремы \ref{th3}
является верным, если $q_{x_0}(r)=\,O\left({\left(
\log{\frac{1}{r}}\right)}^{n-1}\right)$ при $r\rightarrow 0.$

\medskip
Приведём по этому поводу ещё один важный результат, аналог которого
был получен для гомеоморфизмов на плоскости в \cite[лемма~5.1 и
теорема~5.1]{GRY} (насколько нам известно, пространственный случай
как гомеоморфизмов, так и отображений с ветвлением, нигде ранее не
публиковался).

\medskip
\begin{theorem}\label{th4}
{\,\sl Пусть $n\geqslant 2,$ $Q:{\Bbb R}^n\rightarrow[0, \infty],$
$Q(x)\equiv 0$ на ${\Bbb R}^n\setminus D,$ $p\geqslant 1,$ область
$D\subset {\Bbb R}^n$ регулярна, а $D^{\,\prime}\subset {\Bbb R}^n$
ограничена и имеет локально квазиконформную границу, являющуюся
сильно достижимой относительно $p$-модуля. Пусть также отображение
$f:D\rightarrow D^{\,\prime},$ $D^{\,\prime}=f(D),$ является
кольцевым $Q$-отображением относительно $p$-модуля в каждой точке
$x_0\in
\partial D,$ кроме того, $f$ является открытым, дискретным и
замкнутым. Тогда $f$ продолжается до непрерывного отображения
$f:\overline{D}_P\rightarrow \overline{D^{\,\prime}}_P,$
$f(\overline{D}_P)=\overline{D}_P^{\,\prime},$ если выполнено одно
из следующих условий:

1) либо в каждой точке $x_0\in
\partial D$ при некотором $\varepsilon_0=\varepsilon_0(x_0)>0$ и
всех $0<\varepsilon<\varepsilon_0$
\begin{equation}\label{eq10A}
\int\limits_{\varepsilon}^{\varepsilon_0}
\frac{dt}{t^{\frac{n-1}{p-1}}q_{x_0}^{\,\frac{1}{p-1}}(t)}<\infty\,,\qquad
\int\limits_{0}^{\varepsilon_0}
\frac{dt}{t^{\frac{n-1}{p-1}}q_{x_0}^{\,\frac{1}{p-1}}(t)}=\infty\,,
\end{equation}
где
$q_{x_0}(r):=\frac{1}{\omega_{n-1}r^{n-1}}\int\limits_{|x-x_0|=r}Q(x)\,d{\mathcal
H}^{n-1};$

2) либо $Q\in FMO(x_0)$ в каждой точке $x_0\in \partial
D.$ 
}
\end{theorem}

\medskip
Сформулируем теперь наиболее важные результаты, относящиеся ко
второй части настоящей работы. Для этой цели напомним некоторые
определения. Пусть $(X,d)$ и
$\left(X^{\,{\prime}},{d}^{\,{\prime}}\right)$ ~--- метрические
пространства с расстояниями $d$  и ${d}^{\,{\prime}},$
соответственно. Семейство $\frak{F}$ отображений $f:X\rightarrow
{X}^{\,\prime}$ называется {\it равностепенно непрерывным в точке}
$x_0 \in X,$ если для любого $\varepsilon>0$ найдётся $\delta
> 0,$ такое, что ${d}^{\,\prime} \left(f(x),f(x_0)\right)<\varepsilon$ для
всех $f \in \frak{F}$ и  для всех $x\in X$ таких, что
$d(x,x_0)<\delta.$ Говорят, что $\frak{F}$ {\it равностепенно
непрерывно}, если $\frak{F}$ равностепенно непрерывно в каждой точке
из $x_0\in X.$ Всюду далее, если не оговорено противное, $d$ -- одна
из метрик в пространстве простых концов относительно области $D,$
упомянутых выше, а $d^{\,\prime}$ -- евклидова метрика.

\medskip
Для числа $\alpha,$ такого что $\alpha>1,$ областей $D,$
$D^{\,\prime}\subset {\Bbb R}^n,$ $z_1, z_2\in D,$ $z_1\ne z_2,$
$z_1^{\prime},$ $z_2^{\prime}\in D^{\prime}$ и произвольной
измеримой по Лебегу функции $Q(x)$ обозначим символом $\frak{F}_{
\varphi, Q, \alpha}^{z_1, z_2, z_1^{\,\prime}, z_2^{\,\prime}}(D,
D^{\,\prime})$ семейство всех гомеоморфизмов $f:D\rightarrow
D^{\,\prime}$ класса $W_{loc}^{1, \varphi}$ в $D,$
$f(D)=D^{\,\prime},$ таких что $K_{I, \alpha}(x, f)\leqslant Q(x)$ и
\begin{equation*} f(z_1)=z_1^{\prime},\quad
f(z_2)=z_2^{\prime}\,.\end{equation*}
Справедливо следующее утверждение.

\medskip
\begin{theorem}\label{th7}{\,\sl
Пусть $n\geqslant 2,$ $\alpha>1,$ область $D\subset {\Bbb R}^n$
регулярна, а область $D^{\,\prime}\subset {\Bbb R}^n$ имеет локально
квазиконформную границу, которая является сильно достижимой
относительно $\alpha$-модуля. Предположим, $Q\in L_{loc}^1({\Bbb
R}^n),$ что заданная неубывающая функция
$\varphi:[0,\infty)\rightarrow[0,\infty)$ удовлетворяет условию
(\ref{eqOS3.0a}), и что для каждого $x_0\in \overline{D}$ выполнено
одно из следующих условий:

1) либо $Q\in FMO(\overline{D});$

2) либо в каждой точке $x_0\in \overline{D}$ при некотором
$\varepsilon_0=\varepsilon_0(x_0)>0$ и всех
$0<\varepsilon<\varepsilon_0$
$$
\int\limits_{\varepsilon}^{\varepsilon_0}
\frac{dt}{t^{\frac{n-1}{n-\alpha}}q_{x_0}^{\,\frac{1}{\alpha-1}}(t)}<\infty\,,\qquad
\int\limits_{0}^{\varepsilon_0}
\frac{dt}{t^{\frac{n-1}{n-\alpha}}q_{x_0}^{\,\frac{1}{\alpha-1}}(t)}=\infty\,,
$$
где
$q_{x_0}(r):=\frac{1}{\omega_{n-1}r^{n-1}}\int\limits_{|x-x_0|=r}Q(x)\,d{\mathcal
H}^{n-1}.$

Тогда каждый элемент $f\in \frak{F}_{\varphi, Q, \alpha}^{z_1, z_2,
z_1^{\,\prime}, z_2^{\,\prime}}(D, D^{\,\prime})$ продолжается до
непрерывного отображения $\overline f\colon\overline
D_P\rightarrow\overline{D^{\,\prime}}_P$, при этом, семейство
отображений $\frak{F}_{\varphi, Q, \alpha}^{z_1, z_2,
z_1^{\,\prime}, z_2^{\,\prime}}(\overline{D}_P,
\overline{D^{\,\prime}}_P),$ состоящее из всех продолженных таким
образом отображений, является равностепенно непрерывным, а значит, и
нормальным  в $\overline{D}_P$.}
\end{theorem}

\medskip Ещё один вариант теоремы о нормальных семействах
отображений относится к ситуации, когда фиксируется одна, а не две
точки заданной области. По этому поводу напомним ещё одно важное
определение. Согласно \cite{GM}, область $D$ в ${\Bbb R}^n$ будем
называть {\it областью квазиэкстремальной длины}, сокр. $QED$-{\it
облас\-тью}, если
\begin{equation}\label{eq4***}
M(\Gamma(E, F, {\Bbb R}^n))\leqslant A\cdot M(\Gamma(E, F, D))
\end{equation}
для конечного числа $A\geqslant 1$ и всех континуумов $E$ и $F$ в
$D.$ Для областей $D,$ $D^{\,\prime}\subset {\Bbb R}^n,$ $b_0\in D,$
$b_0^{\,\prime}\in D^{\,\prime}$ и произвольной измеримой по Лебегу
функции $Q$ обозначим символом $\frak{F}_{b_0, b_0^{\,\prime},
\varphi, Q}(D, D^{\,\prime})$ семейство всех гомеоморфизмов
$f:D\rightarrow D^{\,\prime}$ класса $W_{loc}^{1, \varphi}$ в $D,$
$f(D)=D^{\,\prime},$ таких что $K_I(x, f)\leqslant Q(x)$ и
$f(b_0)=b_0^{\,\prime}.$ Справедливо следующее утверждение.

\medskip
\begin{theorem}\label{th7A}{\,\sl Пусть область $D$ регулярна, область $D^{\,\prime}$ ограничена,
имеет локально квазиконформную границу и, одновременно, является
$QED$-областью. Предположим, $Q\in L_{loc}^1({\Bbb R}^n),$ заданная
неубывающая функция $\varphi:[0,\infty)\rightarrow[0,\infty)$
удовлетворяет условию (\ref{eqOS3.0a}), и что для каждого $x_0\in
\overline{D}$ выполнено одно из следующих условий:

1) либо $Q\in FMO(\overline{D});$

2) либо в каждой точке $x_0\in \overline{D}$ при некотором
$\varepsilon_0=\varepsilon_0(x_0)>0$
$$
\int\limits_{0}^{\varepsilon_0}
\frac{dt}{tq_{x_0}^{\,\frac{1}{n-1}}(t)}=\infty\,,
$$
где
$q_{x_0}(r):=\frac{1}{\omega_{n-1}r^{n-1}}\int\limits_{|x-x_0|=r}Q(x)\,d{\mathcal
H}^{n-1}.$

Тогда каждый элемент $f\in \frak{F}_{b_0, b_0^{\,\prime}, \varphi,
Q}(D, D^{\,\prime})$ продолжается до непрерывного отображения
$\overline f\colon\overline D_P\rightarrow\overline
{D^{\,\prime}}_P$, при этом, семейство отображений $\frak{F}_{b_0,
b_0^{\,\prime}, \varphi, Q}(\overline{D}_P,
\overline{D^{\,\prime}}_P),$ состоящее из всех продолженных таким
образом отображений, является равностепенно непрерывным, а значит, и
нормальным  в $\overline{D}_P$.}
\end{theorem}

{\bf 2. Вспомогательные сведения.} Дальнейшее изложение и
доказательство теоремы \ref{th3} существенно опираются на аппарат
нак называемых нижних $Q$-гомеоморфизмов (см. \cite[глава~9]{MRSY}).
Говорят, что некоторое свойство $P$ выполнено для {\it $p$-почти
всех поверхностей} области $D,$ если оно имеет место для всех
поверхностей, лежащих в $D,$ кроме, быть может, некоторого их
подсемейства, $p$-модуль которого равен нулю. Будем говорить, что
измеримая по Лебегу функция $\rho:{\Bbb
R}^n\rightarrow\overline{{\Bbb R}^+}$ {\it обобщённо допустима
относительно $p$-модуля} для семейства $\Gamma$ $k$-мерных
поверхностей $S$ в ${\Bbb R}^n,$ сокр. $\rho\in{\rm ext}_p\,{\rm
adm}\,\Gamma,$ если соотношение (\ref{eq8.2.6}) выполнено для
$p$-почти всех поверхностей $S$ семейства $\Gamma.$ Следующий класс
отображений представляет собой обобщение квазиконформных отображений
в смысле кольцевого определения по Герингу (\cite{Ge$_3$}) и
отдельно исследуется (см., напр., \cite[глава~9]{MRSY}). Пусть $D$ и
$D^{\,\prime}$ -- заданные области в ${\Bbb R}^n,$ $n\geqslant 2,$
$x_0\in\overline{D}\setminus\{\infty\}$ и $Q:D\rightarrow(0,\infty)$
-- измеримая по Лебегу функция. Будем говорить, что $f:D\rightarrow
D^{\,\prime}$ -- {\it нижнее $Q$-отображение в точке $x_0$
относительно $p$-модуля,} как только
\begin{equation}\label{eq1A}
M_p(f(\Sigma_{\varepsilon}))\geqslant \inf\limits_{\rho\in{\rm
ext}_p\,{\rm adm}\,\Sigma_{\varepsilon}}\int\limits_{D\cap A(x_0,
\varepsilon, r_0)}\frac{\rho^p(x)}{Q(x)}\,dm(x)
\end{equation}
для каждого кольца $A(x_0, \varepsilon, r_0)=\{x\in {\Bbb R}^n:
\varepsilon<|x-x_0|<r_0\},$ $r_0\in(0,d_0),$ $d_0=\sup\limits_{x\in
D}|x-x_0|,$
где $\Sigma_{\varepsilon}$ обозначает семейство всех пересечений
сфер $S(x_0, r)$ с областью $D,$ $r\in (\varepsilon, r_0).$ Если
$p=n,$ то будем говорить, что $f$ -- нижнее $Q$-отображение в точке
$x_0.$ Будем говорить, что $f$ нижнее $Q$-отображение относительно
$p$-модуля в $A\subset \overline{D},$ если соотношение (\ref{eq1A})
имеет место для каждого $x_0\in A.$

\medskip Имеет место
следующее утверждение, которое может быть доказано аналогично
теореме 9.2 в \cite{MRSY}, и потому опускается.

\medskip
\begin{lemma}\label{lem4}{\,\sl
Пусть $D,$  $D^{\,\prime}\subset\overline{{\Bbb R}^n},$
$x_0\in\overline{D}\setminus\{\infty\}$ и $Q$ -- измеримая по Лебегу
функция. Отображение $f:D\rightarrow D^{\,\prime}$ является нижним
$Q$-отображением относительно $p$-модуля в точке $x_0,$ $p>n-1,$
тогда и только тогда, когда
%
$M_p(f(\Sigma_{\varepsilon}))\geqslant\int\limits_{\varepsilon}^{r_0}
\frac{dr}{\Vert\,Q\Vert_{s}(r)}\quad\forall\
\varepsilon\in(0,r_0)\,,\ r_0\in(0,d_0),$ $d_0=\sup\limits_{x\in
D}|x-x_0|,$
%
$s=\frac{n-1}{p-n+1},$ где, как и выше, $\Sigma_{\varepsilon}$
обозначает семейство всех пересечений сфер $S(x_0, r)$ с областью
$D,$ $r\in (\varepsilon, r_0),$
$ \Vert
Q\Vert_{s}(r)=\left(\int\limits_{D(x_0,r)}Q^{s}(x)\,d{\mathcal{A}}\right)^{\frac{1}{s}}$
-- $L_{s}$-норма функции $Q$ над сферой $D(x_0,r)=\{x\in D:
|x-x_0|=r\}=D\cap S(x_0,r)$. }
\end{lemma}

\medskip
Следующие важные сведения, касающиеся ёмкости пары множеств
относительно области, могут быть найдены в работе В.~Цимера
\cite{Zi}. Пусть $G$ -- ограниченная область в ${\Bbb R}^n$ и $C_{0}
, C_{1}$ -- непересекающиеся компактные множества, лежащие в
замыкании $G.$ Полагаем  $R=G \setminus (C_{0} \cup C_{1})$ и
$R^{\,*}=R \cup C_{0}\cup C_{1},$ тогда {\it $p$-ёмкостью пары
$C_{0}, C_{1}$ относительно замыкания $G$} называется величина
$C_p[G, C_{0}, C_{1}] = \inf \int\limits_{R} \vert \nabla u
\vert^{p}\ dm(x),$
где точная нижняя грань берётся по всем функциям $u,$ непрерывным в
$R^{\,*},$ $u\in ACL(R),$ таким что $u=1$ на $C_{1}$ и $u=0$ на
$C_{0}.$ Указанные функции будем называть {\it допустимыми} для
величины $C_p [G, C_{0}, C_{1}].$ Мы будем говорить, что  {\it
множество $\sigma \subset {\Bbb R}^n$ разделяет $C_{0}$ и $C_{1}$ в
$R^{\,*}$}, если $\sigma \cap R$ замкнуто в $R$ и найдутся
непересекающиеся множества $A$ и $B,$ являющиеся открытыми в
$R^{\,*} \setminus \sigma,$ такие что $R^{\,*} \setminus \sigma =
A\cup B,$ $C_{0}\subset A$ и $C_{1} \subset B.$ Пусть $\Sigma$
обозначает класс всех множеств, разделяющих $C_{0}$ и $C_{1}$ в
$R^{\,*}.$ Для числа $p^{\prime} = p/(p-1)$ определим величину
%
$$\widetilde{M}_{p^{\prime}}(\Sigma)=\inf\limits_{\rho\in
\widetilde{\rm adm} \Sigma} \int\limits_{{\Bbb
R}^n}\rho^{\,p^{\prime}}dm(x)\,,$$
%
где запись $\rho\in \widetilde{\rm adm}\,\Sigma$ означает, что
$\rho$ -- неотрицательная борелевская функция в ${\Bbb R}^n$ такая,
что
\begin{equation} \label{eq13.4.13}
\int\limits_{\sigma \cap R}\rho d{\mathcal H}^{n-1} \geqslant
1\quad\forall\, \sigma \in \Sigma\,.
\end{equation}
Заметим, что согласно результата Цимера
\begin{equation}\label{eq3}
\widetilde{M}_{p^{\,\prime}}(\Sigma)=C_p[G , C_{0} ,
C_{1}]^{\,-1/(p-1)}\,,
\end{equation}
см. \cite[теорема~3.13]{Zi} при $p=n$ и \cite[с.~50]{Zi$_1$} при
$1<p<\infty.$ Заметим также, что согласно результата Хессе
\begin{equation}\label{eq4}
M_p(\Gamma(E, F, D))= C_p[D, E, F]\,,
\end{equation}
см.~\cite[теорема~5.5]{Hes}.

\medskip
Для отображения $f:D\,\rightarrow\,{\Bbb R}^n,$ множества $E\subset
D$ и $y\,\in\,{\Bbb R}^n,$  определим {\it функцию кратности $N(y,
f, E)$} как число прообразов точки $y$ во множестве $E,$ т.е.
\begin{equation}\label{eq1.7A}
N(y, f, E)\,=\,{\rm card}\,\left\{x\in E: f(x)=y\right\}\,,\quad
%
N(f, E)\,=\,\sup\limits_{y\in{\Bbb R}^n}\,N(y, f, E)\,.
\end{equation}
Пусть $I$ -- открытый, полуоткрытый или замкнутый интервал в ${\Bbb
R}.$ {\it Носителем} кривой $\alpha:I\rightarrow {\Bbb R}^n$
называется множество
$$|\alpha|:=\{x\in {\Bbb R}^n: \exists\,t\in I:
\alpha(t)=x\}\,.$$ Иногда, если недоразумение невозможно, кривая
$\alpha$ и её носитель $|\alpha|$ отождествляются. Имеет место
следующее утверждение, см., напр., \cite[лемма~3.7]{Vu$_1$}.

\medskip
\begin{proposition}\label{pr7}
{\sl\, Пусть $f:D\rightarrow {\Bbb R}^n$ -- открытое дискретное и
замкнутое отображение, $\beta:[a, b)\rightarrow f(D)$ --
произвольная кривая и $l=\sup\limits_{y\in {\Bbb R}^n}N(y, f, D).$
Тогда найдутся кривые $\alpha_j:[a, b)\rightarrow D,$ $1\leqslant
j\leqslant l,$ со следующим свойством:

(1)\qquad $f\circ\alpha_j=\beta,$ (2)\quad ${\rm card}\,\{j:
\alpha_j(t)=x\}=|i(x, f)|$ для всех $x\in f^{\,-1}(|\beta|)$ и всех
$t\in [a, b)$ (где $i(x, f)$ -- локальный топологический индекс
отображения $f$ в точке $x$), и (3)\quad
$\bigcup\limits_{j=1}^l|\alpha_j|=f^{\,-1}(|\beta|).$ }
\end{proposition}

\medskip
Говорят, что семейство кривых $\Gamma_1$ {\it минорируется}
семейством $\Gamma_2,$ пишем $\Gamma_1\,>\,\Gamma_2,$
если для каждой кривой $\gamma\,\in\,\Gamma_1$ существует подкривая,
которая принадлежит семейству $\Gamma_2.$
В этом случае,
\begin{equation}\label{eq32*A}
\Gamma_1
> \Gamma_2 \quad \Rightarrow \quad M_p(\Gamma_1)\leqslant M_p(\Gamma_2)
\end{equation} (см. \cite[теорема~6.4, гл.~I]{Va}).

\medskip
{\bf 3. О продолжении нижних $Q$-отображений на границу.} В
дальнейшем нам понадобится следующее вспомогательное утверждение
(см., напр., \cite[лемма~7.4, гл.~7]{MRSY} и \cite[лемма~2.2]{Sal}
при $p\ne n.$

\medskip
\begin{proposition}\label{pr1A}
{\,\sl Пусть  $x_0 \in {\Bbb R}^n,$ $Q(x)$ -- измеримая по Лебегу
функция, $Q\in L_{loc}^1({\Bbb R}^n).$ Полагаем
$A:=A(r_1,r_2,x_0)=\{ x\,\in\,{\Bbb R}^n : r_1<|x-x_0|<r_2\}$ и
$\eta_0(r)=\frac{1}{Ir^{\frac{n-1}{p-1}}q_{x_0}^{\frac{1}{p-1}}(r)},$
где $I:=I=I(x_0,r_1,r_2)=\int\limits_{r_1}^{r_2}\
\frac{dr}{r^{\frac{n-1}{p-1}}q_{x_0}^{\frac{1}{p-1}}(r)}$ и
$q_{x_0}(r):=\frac{1}{\omega_{n-1}r^{n-1}}\int\limits_{|x-x_0|=r}Q(x)\,d{\mathcal
H}^{n-1}$ -- среднее интегральное значение функции $Q$ над сферой
$S(x_0, r).$ Тогда
$$
\frac{\omega_{n-1}}{I^{p-1}}=\int\limits_{A} Q(x)\cdot
\eta_0^p(|x-x_0|)\ dm(x)\leqslant\int\limits_{A} Q(x)\cdot
\eta^p(|x-x_0|)\ dm(x)
$$
для любой измеримой по Лебегу функции $\eta :(r_1,r_2)\rightarrow
[0,\infty]$ такой, что
$\int\limits_{r_1}^{r_2}\eta(r)dr=1. $ }
\end{proposition}

Справедливо следующее утверждение.

\medskip
\begin{theorem}\label{th2}
{\,\sl Пусть $n\geqslant 2,$ $p>n-1,$ область $D\subset {\Bbb R}^n$
регулярна, а $D^{\,\prime}\subset {\Bbb R}^n$ ограничена и имеет
локально квазиконформную границу, являющуюся сильно достижимой
относительно $\alpha$-модуля, $\alpha:=\frac{p}{p-n+1}.$ Пусть также
отображение $f:D\rightarrow D^{\,\prime},$ $D^{\,\prime}=f(D),$
является нижним $Q$-отображением в каждой точке $x_0\in \partial D$
относительно $p$-модуля, кроме того, $f$ является открытым,
дискретным и замкнутым. Тогда $f$ продолжается до непрерывного
отображения $f:\overline{D}_P\rightarrow \overline{D^{\,\prime}}_P,$
$f(\overline{D}_P)=\overline{D^{\,\prime}}_P,$ если выполнено одно
из следующих условий:

\medskip
1) либо в каждой точке $x_0\in \partial D$ при некотором
$\varepsilon_0=\varepsilon_0(x_0)>0$ и всех
$0<\varepsilon<\varepsilon_0$ выполнены следующие условия:
\begin{equation}\label{eq10}
\int\limits_{\varepsilon}^{\varepsilon_0}
\frac{dt}{t^{\frac{n-1}{\alpha-1}}\widetilde{q}_{x_0}^{\,\frac{1}{\alpha-1}}(t)}<\infty\,,\qquad
\int\limits_{0}^{\varepsilon_0}
\frac{dt}{t^{\frac{n-1}{\alpha-1}}\widetilde{q}_{x_0}^{\,\frac{1}{\alpha-1}}(t)}=\infty\,,
\end{equation}
где $\alpha=\frac{p}{p-n+1},$
$\widetilde{q}_{x_0}(r):=\frac{1}{\omega_{n-1}r^{n-1}}\int\limits_{|x-x_0|=r}Q^{\frac{n-1}{p-n+1}}(x)\,d{\mathcal
H}^{n-1}$ обозначает среднее интегральное значение функции
$Q^{\frac{n-1}{p-n+1}}(x)$ над сферой $S(x_0, r);$

2) либо $Q^{\frac{n-1}{p-n+1}}\in FMO(\partial D).$ }
\end{theorem}

\begin{proof} Докажем вначале, что $f$ имеет непрерывное продолжение
$f:\overline{D}_P\rightarrow \overline{D^{\,\prime}}_P.$ Рассмотрим
прежде всего случай 1), т.е., когда имеют место соотношения
(\ref{eq10}). Так как $D^{\,\prime}$ имеет локально квазиконформную
границу, то $\overline{D^{\,\prime}}_P=\overline{D^{\,\prime}}$ (см.
\cite[теорема~4.1]{Na}). В силу метризуемости пространства
$\overline{D}_P$ достаточно доказать, что для каждого простого конца
$P$ области $D$ предельное множество
$$L=C(f, P):=\left\{y\in{{\Bbb
R}^n}:y=\lim\limits_{m\rightarrow\infty}f(x_m),x_m\rightarrow
P,x_m\in D\right\}$$ состоит из единственной точки $y_0\in\partial
D^{\,\prime}$. (Здесь мы говорим, что последовательность точек
$x_m\in D$, $m=1,2,\ldots$, {\it сходится к концу} $K$, пишем
$x_m\rightarrow P$ при $m\rightarrow\infty,$ если для каждой цепи
$\{\sigma_m\}$ в $K$ и каждой области $d_m$ все точки $x_m$, за
исключением, быть может, конечного числа, принадлежат $d_m$, см.,
напр., \cite[$\S\,3$]{KR}).

Заметим, что $L\ne\varnothing$ в силу компактности множества
$\overline{D^{\,\prime}}$, и $L$ является подмножеством $\partial
D^{\,\prime}$ ввиду замкнутости отображения $f$ (см.
\cite[теорема~3.3]{Vu$_1$}). Предположим, что существуют, по крайней
мере, две точки $y_0$ и $z_0\in L$. Положим $U=B(y_0,r_0)$, где
$0<r_0<|y_0-z_0|$.

В силу \cite[лемма~2]{KR} каждый простой конец $P$ регулярной
области $D$ в ${\Bbb R}^n$, $n\geqslant 2,$ содержит цепь разрезов
$\sigma_m$, лежащую на сферах $S_m$ с центром в некоторой точке
$x_0\in\partial D$ и евклидовыми радиусами $r_m\rightarrow 0$ при
$m\rightarrow\infty$. Пусть $D_m$ -- области, ассоциированные с
разрезами $\sigma_m$, $m=1,2,\ldots$. Тогда существуют точки $y_m$ и
$z_m$ в областях $D_{m}^{\,\prime}=f(D_{m})$, такие что
$|y_0-y_m|<r_0$ и $|y_0-z_m|>r_0$ и, кроме того, $y_m\rightarrow
y_0$ и $z_m\rightarrow z_0$ при $m\rightarrow\infty$.
Соответственно, найдутся $x_m$ и $x_m^{\,\prime}$ в области $D_m,$
такие что $f(x_m)=y_m$ и $f(x_m^{\,\prime})=z_m.$ Соединим точки
$x_m$ и $y_m$ кривой $\gamma_m,$ целиком лежащей в области $D_m.$
Пусть $C_m$ -- образ этой кривой при отображении $f$ в
$D^{\,\prime}.$ Заметим, что $\partial U\cap |C_m|\ne\varnothing$
ввиду \cite[теорема~1.I.5, $\S\,46$]{Ku} (как обычно, $|C_m|$
обозначает носитель кривой $C_m$).

В силу определения сильно достижимой границы относительно
$\alpha$-модуля существует компакт $E\subset D^{\,\prime}$ и число
$\delta>0$, такие, что
\begin{equation}\label{eq1}
M_{\alpha}(\Gamma(E, |C_m|, D^{\,\prime}))\geqslant\delta
\end{equation}
для всех достаточно больших $m$.

Без ограничения общности можем считать, что последнее условие
выполнено для всех $m=1,2,\ldots$. Заметим, что $C=f^{-1}(E)$
является компактным подмножеством области $D$ ввиду замкнутости
отображения $f$ (см. \cite[теорема~3.3]{Vu$_1$}), поэтому, поскольку
$I(P)=\bigcap\limits_{m=1}\limits^{\infty}\overline{D_m}\subset
\partial D$ (см. \cite[предложение~1]{KR}), то не ограничивая общности рассуждений, можно считать, что
$C\cap\overline{D_m}=\varnothing$ для каждого $m\in {\Bbb N.}$
Положим $\delta_0:={\rm dist}\,(x_0, C).$ Не ограничивая общности,
уменьшая $\varepsilon_0,$ если это необходимо, можно считать, что
$\varepsilon_0<\delta_0.$

Пусть $\Gamma_m$ -- семейство всех кривых в $D,$ соединяющих $C$ и
$\sigma_m$, $m=1,2,\ldots$. Заметим, что $\Gamma(|\gamma_m|, C,
D)>\Gamma_m$ ввиду \cite[теорема~1.I.5, $\S\,46$]{Ku}, так что
$f(\Gamma(|\gamma_m|, C, D))>f(\Gamma_m)$ и ввиду (\ref{eq32*A})
\begin{equation}\label{eq1D}
M_{\alpha}(f(\Gamma(|\gamma_m|, C, D)))\leqslant
M_{\alpha}(f(\Gamma_m))\,.
\end{equation}
Оценим $M_{\alpha}(f(\Gamma(|\gamma_m|, C, D)))$ в формуле
(\ref{eq1D}) снизу.
Пусть кривая $\beta:[0,1)\rightarrow D^{\,\prime}$ такова, что
$\beta(0)\in |C_m|$ и $\beta(t)\rightarrow p\in E$ при $t\rightarrow
1-0,$ где $p$ -- некоторый фиксированный элемент множества $E.$
Тогда для кривой $\beta$ ввиду предложения \ref{pr7} найдётся другая
кривая $\gamma:[0, 1)\rightarrow D$ с началом в $|\gamma_m|,$ такая,
что $f\circ\gamma=\beta.$ Поскольку $f$ -- замкнутое отображение, то
оно сохраняет границу (см. \cite[теорема~3.3]{Vu$_1$}) и, значит,
кривая $\gamma$ лежит в $D$ вместе со своим замыканием. Более того,
в силу дискретности отображения $f$ кривая $\gamma$ может быть
продолжена до кривой $\overline{\gamma}:[0, 1]\rightarrow D.$
Заметим, что по определению $\beta(1)\in E,$ так что
$\overline{\gamma}(1)\in C$ по определению множества $C.$ Значит,
$\overline{\gamma}\in \Gamma(|\gamma_m|, C, D).$ Рассмотрим
семейство $\Gamma_m^*,$ состоящее из всех возможных таких кривых
$\overline{\gamma},$ тогда $\Gamma_m^*\subset\Gamma(|\gamma_m|, C,
D)$ и, одновременно, $f(\Gamma_m^*)=\Gamma(E, |C_m|, D^{\,\prime}).$
Тогда
\begin{equation}\label{eq5}
M_{\alpha}(\Gamma(E, |C_m|,
D^{\,\prime}))=M_{\alpha}(f(\Gamma^*_m))\leqslant
M_{\alpha}(f(\Gamma(|\gamma_m|, C, D)))\,.
\end{equation}
Из (\ref{eq1}), (\ref{eq1D}) и (\ref{eq5}) вытекает, что
\begin{equation}\label{eq7}
M_{\alpha}(f(\Gamma_{m}))\geqslant\delta
\end{equation}
для всех $m=1,2,\ldots$. Заметим, что
$f(\Gamma_m)\subset\Gamma(f(\sigma_m), E, D^{\,\prime}),$ поэтому из
(\ref{eq7}) вытекает, что
\begin{equation}\label{eq7B}
M_{\alpha}(\Gamma(f(\sigma_m), E,
D^{\,\prime}))\geqslant\delta\quad\forall\quad m=1,2,\ldots .
\end{equation}

\medskip
Оценим теперь величину $M_{\alpha}(\Gamma(f(\sigma_m), E,
D^{\,\prime}))$ сверху. Для этого подберём подходящим для нас
способом систему разделяющих множеств для $E$ и $f(\sigma_m)$ и
воспользуемся определением нижнего $Q$-отображения.

\medskip
Заметим, прежде всего, что множества $E$ и $\overline{f(B(x_0,
r)\cap D)}$ не пересекаются при любом $r\in (0, \varepsilon_0).$
Предположим противное, а именно, что найдётся $\zeta_0\in E\cap
\overline{f(B(x_0, r)\cap D)}.$ Тогда
$\zeta_0=\lim\limits_{k\rightarrow\infty} \zeta_k,$ где $\zeta_k\in
f(B(x_0, r)\cap D).$ Отсюда $\zeta_k=f(\xi_k),$ $\xi_k\in B(x_0,
r)\cap D.$ Так как $\overline{D}$ -- компакт, то из
последовательности $\xi_k$ можно выделить сходящуюся
подпоследовательность $\xi_{k_l}\rightarrow \xi_0\in
\overline{B(x_0, r)\cap D}.$ Случай $\xi_0\in \partial D$
невозможен, поскольку $f$ -- замкнутое отображение и, значит,
сохраняет границу: $C(f,
\partial D)\subset \partial f(D),$ но у нас $\zeta_0$ -- внутренняя
точка $D^{\,\prime}.$ Пусть $\xi_0$ -- внутренняя точка $D.$ По
непрерывности отображения $f$ имеем $f(\xi_0)=\zeta_0.$ Но тогда
одновременно $\xi_0\in B(x_0, \varepsilon_0)\cap D$ и $\xi_0\in
f^{\,-1}(E),$ что противоречит выбору $\varepsilon_0.$ Таким
образом, $E\cap \overline{f(B(x_0, r)\cap D)}=\varnothing$ и,
значит,
\begin{equation}\label{eq8}
E\subset D^{\,\prime}\setminus \overline{f(B(x_0, r)\cap D)}\,, r\in
(0, \varepsilon_0)\,.
\end{equation}
Из (\ref{eq8}), в частности, вытекает, что множества $E$ и
$f(\sigma_m)$ не пересекаются.

Заметим также, что при произвольном $r\in (r_m, \varepsilon_0)$
множество $A_r:=\partial (f(B(x_0, r)\cap D))\cap D^{\,\prime}$
отделяет $E$ и $f(\sigma_m)$ в $D^{\,\prime}.$ Действительно,
$$D^{\,\prime}=B_r\cup A_r\cup C_r\qquad\forall\quad
r\in (r_m, \varepsilon_0)\,,$$
где множества $B_r:=f(B(x_0, r)\cap D)$ и
$C_r:=D^{\,\prime}\setminus \overline{f(B(x_0, r)\cap D)}$ открыты в
$D^{\,\prime},$ $f(\sigma_m)\subset B_r,$ $E\subset C_r$ и $A_r$
замкнуто в $D^{\,\prime}.$

\medskip
Пусть $\Sigma_m$ -- семейство всех множеств, отделяющих
$f(\sigma_m)$ от $E$ в $D^{\,\prime}.$ Поскольку $f$ -- открытое
замкнутое отображение, мы получим, что
\begin{equation}\label{eq7A}
(\partial f(B(x_0, r)\cap D))\cap D^{\,\prime}\subset f(S(x_0,
r)\cap D), r>0.
\end{equation}
Действительно, пусть $\zeta_0\in (\partial f(B(x_0, r)\cap D))\cap
D^{\,\prime}.$ Тогда найдётся последовательность $\zeta_k\in
f(B(x_0, r)\cap D)$ такая, что $\zeta_k\rightarrow \zeta_0$ при
$k\rightarrow \infty,$ где $\zeta_k=f(\xi_k),$ $\xi_k\in B(x_0,
r)\cap D.$ Не ограничивая общности рассуждений, можно считать, что
$\xi_k\rightarrow \xi_0$ при $k\rightarrow\infty.$ Заметим, что
случай $\xi_0\in \partial D$ невозможен, поскольку в этом случае
$\zeta_0\in C(f,
\partial D),$ что противоречит замкнутости отображения $f.$ Тогда
$\xi_0\in D.$ Возможны две ситуации: 1) $\xi_0\in B(x_0 , r)\cap D$
и 2) $\xi_0\in S(x_0 , r)\cap D.$ Заметим, что случай 1) невозможен,
поскольку, в этом случае, $f(\xi_0)=\zeta_0$ и $\zeta_0$ --
внутренняя точка множества $f(B(x_0, r)\cap D),$ что противоречит
выбору $\zeta_0.$ Таким образом, включение (\ref{eq7A}) установлено.

\medskip
Здесь и далее объединения вида $\bigcup\limits_{r\in (r_1, r_2)}
\partial f(B(x_0, r)\cap D)\cap D^{\,\prime}$ понимаются как семейства множеств. Пусть $\rho^{n-1}\in \widetilde{{\rm
adm}}\bigcup\limits_{r\in (r_m, \varepsilon_0)}
\partial f(B(x_0, r)\cap D)\cap D^{\,\prime}$ в смысле соотношения (\ref{eq13.4.13}), тогда также
$\rho\in {\rm adm}\bigcup\limits_{r\in (r_m, \varepsilon_0)}
\partial f(B(x_0, r)\cap D)\cap D^{\,\prime}$ в смысле соотношения (\ref{eq8.2.6}) при
$k=n-1.$ Ввиду (\ref{eq7A}) мы получим, что $\rho\in {\rm
adm}\bigcup\limits_{r\in (r_m, \varepsilon_0)} f(S(x_0, r)\cap D)$
и, следовательно, так как $\widetilde{M}_q(\Sigma_m)\geqslant
M_{q(n-1)}(\Sigma_m)$ при произвольном $q\geqslant 1,$ то
$$\widetilde{M}_{p/(n-1)}(\Sigma_m)\geqslant$$
\begin{equation}\label{eq5A}\geqslant
\widetilde{M}_{p/(n-1)}\left(\bigcup\limits_{r\in (r_m,
\varepsilon_0)}
\partial f(B(x_0, r)\cap D)\cap D^{\,\prime}\right)\geqslant
\end{equation}
$$\geqslant\widetilde{M}_{p/(n-1)}
\left(\bigcup\limits_{r\in (r_m, \varepsilon_0)} f(S(x_0, r)\cap
D)\right)\geqslant M_{p}\left(\bigcup\limits_{r\in (r_m,
\varepsilon_0)} f(S(x_0, r)\cap D)\right)\,.$$

Однако, ввиду (\ref{eq3}) и (\ref{eq4}), учитывая, что $p>n-1,$
имеем
\begin{equation}\label{eq6A}
\widetilde{M}_{p/(n-1)}(\Sigma_m)=\frac{1}{(M_{\alpha}(\Gamma(f(\sigma_m),
E, D^{\,\prime})))^{1/(\alpha-1)}}\,.
\end{equation}
По лемме \ref{lem4}
$$M_{p}\left(\bigcup\limits_{r\in (r_m, \varepsilon_0)} f(S(x_0,
r)\cap D)\right)\geqslant
$$
\begin{equation}\label{eq8A}
\geqslant \int\limits_{r_m}^{\varepsilon_0}
\frac{dr}{\Vert\,Q\Vert_{s}(r)}= \int\limits_{r_m}^{\varepsilon_0}
\frac{dt}{\omega^{\frac{p-n+1}{n-1}}_{n-1}
t^{\frac{n-1}{\alpha-1}}\widetilde{q}_{x_0}^{\,\frac{1}{\alpha-1}}(t)}\quad\forall\,\,
m\in {\Bbb N}\,, s=\frac{n-1}{p-n+1}\,,\end{equation}
$\alpha=p/(p-n+1),$ где
$\Vert
Q\Vert_{s}(r)=\left(\int\limits_{D(x_0,r)}Q^{s}(x)\,d{\mathcal{A}}\right)^{\frac{1}{s}}$
-- $L_{s}$-норма функции $Q$ над сферой $D(x_0,r):=S(x_0,r)\cap D.$
Из условий (\ref{eq10}) вытекает, что
$\int\limits_{r_m}^{\varepsilon_0}
\frac{dt}{\omega^{\frac{p-n+1}{n-1}}_{n-1}
t^{\frac{n-1}{\alpha-1}}\widetilde{q}_{x_0}^{\,\frac{1}{\alpha-1}}(t)}\rightarrow\infty$
при $m\rightarrow\infty.$

Из соотношений (\ref{eq5A}) и (\ref{eq8A}) следует, что
$\widetilde{M}_{p/(n-1)}(\Sigma_m)\rightarrow\infty$ при
$m\rightarrow\infty,$ однако, в таком случае, из (\ref{eq6A})
следует, что $M_{\alpha}(\Gamma(f(\sigma_m), E,
D^{\,\prime}))\rightarrow 0$ при $m\rightarrow\infty,$ что
противоречит неравенству (\ref{eq7B}). Полученное противоречие
опровергает предположение, что предельное множество $C(f, P)$
состоит более чем из одной точки.

\medskip
Рассмотрим теперь случай 2), а именно, пусть теперь $Q^s\in
FMO(\partial D),$ $s=(n-1)/(p-n+1).$ Покажем, что в этом случае
выполнено второе условие в (\ref{eq10}). Для этой цели воспользуемся
предложением \ref{pr1A}. Согласно этому предложению для любой
неотрицательной измеримой функции $\eta:(\varepsilon,
\varepsilon_0)\rightarrow [0, \infty,]$ удовлетворяющей условию
$\int\limits_{\varepsilon}^{\varepsilon_0}\eta(t)dt=1,$ выполнено
неравенство
\begin{equation}\label{eq12B} \frac{\omega_{n-1}}{J^{\alpha-1}}\leqslant
\int\limits_{A(x_0, \varepsilon, \varepsilon_0)} Q^s(x)\cdot
\eta^{\alpha}(|x-x_0|)dm(x)\,,
\end{equation}
где $s=(n-1)/(p-n+1),$ $J:=J(x_0, \varepsilon,
\varepsilon_0):=\int\limits_{\varepsilon}^{\varepsilon_0}\frac{dr}{r^{\frac{n-1}{\alpha-1}}\
\widetilde{q}_{x_0}^{\frac{1}{\alpha-1}}(r)},$ а
$\widetilde{q}_{x_0}(r)$ -- среднее значение функции $Q^s$ над
$S(x_0, r)\cap D.$ Положим
$\psi(t):=\frac{1}{\left(t\,\log{\frac1t}\right)^{n/{\alpha}}},$
$I(\varepsilon,
\varepsilon_0):=\int\limits_{\varepsilon}^{\varepsilon_0}\psi(t) dt
\geqslant \log{\frac{\log{\frac{1}
{\varepsilon}}}{\log{\frac{1}{\varepsilon_0}}}}$ и
$\eta(t):=\psi(t)/I(\varepsilon, \varepsilon_0).$ Заметим, что
$\int\limits_{\varepsilon}^{\varepsilon_0}\eta(t)dt=1,$ кроме того,
для функций класса $FMO,$ как известно,
\begin{equation}\label{eq31*}
\int\limits_{\varepsilon<|x|<{e_0}}\frac{Q^s(x+x_0)\, dm(x)}
{\left(|x| \log \frac{1}{|x|}\right)^n} = O \left(\log\log
\frac{1}{\varepsilon}\right)
\end{equation}
при  $\varepsilon \rightarrow 0 $ и для некоторого $e_0>0,$ $e_0
\leqslant {\rm dist}\,\left(0,\partial D\right).$ Исходя из
(\ref{eq31*}) правая часть соотношения (\ref{eq12B}) стремится к 0
при $\varepsilon\rightarrow 0$ и выбранной функции $\eta.$ Но тогда
из (\ref{eq12B}) неизбежно следует, что
$\int\limits_{\varepsilon}^{\varepsilon_0}\frac{dr}{r^{\frac{n-1}{\alpha-1}}\
\widetilde{q}_{x_0}^{\frac{1}{\alpha-1}}(r)}\rightarrow \infty$ при
$\varepsilon\rightarrow 0.$ Повторяя рассуждения от начала
доказательства до соотношений (\ref{eq5A}) и (\ref{eq8A}), мы
заключаем из этих соотношений, что снова
$\widetilde{M}_{p/(n-1)}(\Sigma_m)\rightarrow\infty$ при
$m\rightarrow\infty.$ Однако, в таком случае, из (\ref{eq6A})
следует, что $M_{\alpha}(\Gamma(f(\sigma_m), E,
D^{\,\prime}))\rightarrow 0$ при $m\rightarrow\infty,$ что
противоречит неравенству (\ref{eq7B}). Полученное противоречие
опровергает предположение, что предельное множество $C(f, P)$
состоит более чем из одной точки. Таким образом, утверждение теоремы
о возможности непрерывного продолжения отображения до отображения
$f:\overline{D}_P\rightarrow \overline{D^{\,\prime}}_P$ в случае 2)
также установлено.

\medskip
Для завершения доказательства необходимо показать равенство
$f(\overline{D}_P)=\overline{D^{\,\prime}}.$ Очевидно,
$f(\overline{D}_P)\subset\overline{D^{\,\prime}}.$ Покажем обратное
включение. Пусть $\zeta_0\in \overline{D^{\,\prime}}.$ Если
$\zeta_0$ -- внутренняя точка области $D^{\,\prime},$ то, очевидно,
существует $\xi_0\in D$ так, что $f(\xi_0)=\zeta_0$ и, значит,
$\zeta_0\in f(D).$ Пусть теперь $\zeta_0\in
\partial D^{\,\prime},$ тогда найдётся последовательность $\zeta_m\in D^{\,\prime},$
$\zeta_m=f(\xi_m),$ $\xi_m\in D,$ такая, что $\zeta_m\rightarrow
\zeta_0$ при $m\rightarrow\infty.$ Поскольку $\overline{D}_P$ --
компакт (см. замечания, сделанные перед формулировкой теоремы
\ref{th3}), то можно считать, что $\xi_m\rightarrow P_0,$ где $P_0$
-- некоторый простой конец в $\overline{D}_P.$ Тогда также
$\zeta_0\in f(\overline{D}_P).$ Включение
$\overline{D^{\,\prime}}\subset f(\overline{D}_P)$ доказано и,
значит, $f(\overline{D}_P)=\overline{D^{\,\prime}}.$ Теорема
доказана.~$\Box$
\end{proof}

\medskip
Доказательство следующей леммы аналогично доказательству
\cite[теорема~5]{KRSS} и потому опускается.

\medskip
\begin{lemma}{}\label{thOS4.1} {\sl\, Пусть $D$ -- область в ${\Bbb R}^n,$
$n\geqslant 2,$ $\varphi:(0,\infty)\rightarrow (0,\infty)$ --
неубывающая функция, удовлетворяющая условию (\ref{eqOS3.0a}).
Если $p>n-1,$ то каждое открытое дискретное отображение
$f:D\rightarrow {\Bbb R}^n$ с конечным искажением класса
$W^{1,\varphi}_{loc}$ такое, что $N(f, D)<\infty,$ является нижним
$Q$-отображением относительно $p$-модуля в каждой точке
$x_0\in\overline{D}$ при
$$Q(x)=N(f, D)\cdot K^{\frac{p-n+1}{n-1}}_{I, \alpha}(x, f),$$
$\alpha:=\frac{p}{p-n+1},$ где внутренняя дилатация $K_{I,\alpha}(x,
f)$ отображения $f$ в точке $x$ порядка $\alpha$ определена
соотношением (\ref{eq0.1.1A}), а кратность $N(f, D)$ определена
вторым соотношением в (\ref{eq1.7A}).}
\end{lemma}

\medskip
{\it Доказательство теоремы \ref{th3}}. По лемме \ref{thOS4.1}
отображение $f$ в каждой точке $x_0\in D$ является нижним
$Q$-отображением относительно $p$-модуля в каждой точке
$x_0\in\overline{D}$ при $Q(x)=N(f, D)\cdot
K^{\frac{p-n+1}{n-1}}_{I, \alpha}(x, f),$ $\alpha:=\frac{p}{p-n+1}$
(т.е., $p=\frac{\alpha(n-1)}{\alpha-1}>n-1$), где внутренняя
дилатация $K_{I,\alpha}(x, f)$ отображения $f$ в точке $x$ порядка
$\alpha$ определена соотношением (\ref{eq0.1.1A}), а кратность $N(f,
D)$ определена вторым соотношением в (\ref{eq1.7A}). Тогда
необходимое заключение вытекает из теоремы \ref{th2}, а также того
факта, что максимальная кратность $N(f, D)$ замкнутого открытого
дискретного отображения $f$ конечна (см., напр.,
\cite[лемма~3.3]{MS}).~$\Box$

\medskip
{\bf 4. О граничном поведении ещё одного класса отображений.}
Следующее определение восходит к О. Мартио, В. Рязанову, У. Сребро и
Э. Якубову и имеет важное значение при изучении классов Соболева
(см. \cite{Sev$_1$}). Здесь и далее
\begin{equation}\label{eq49***}
A(x_0, r_1, r_2): =\left\{ x\,\in\,{\Bbb R}^n:
r_1<|x-x_0|<r_2\right\}\,.
\end{equation}
Введём в рассмотрение следующую конструкцию, см. \cite[разд. 7.6
гл~7]{MRSY}.
Пусть $p\geqslant 1,$ $Q$ -- заданная измеримая по Лебегу функция.
Говорят, что отображение $f:D\rightarrow \overline{{\Bbb R}^n}$ есть
{\it кольцевое $Q$-отоб\-ра\-же\-ние в точке $x_0\in \overline{D}$
относительно $p$-модуля,} $x_0\ne \infty,$ если для некоторого
$r_0=r(x_0)$ и произвольных сферического кольца $A=A(x_0,r_1,r_2),$
центрированного в точке $x_0,$ радиусов: $r_1,$ $r_2,$ $0<r_1<r_2<
r_0=r(x_0),$ и любых континуумов $E_1\subset \overline{B(x_0,
r_1)}\cap D,$ $E_2\subset \left(\overline{{\Bbb R}^n}\setminus
B(x_0, r_2)\right)\cap D,$ отображение $f$ удовлетворяет соотношению
\begin{equation}\label{eq2*!}
 M_p\left(f\left(\Gamma\left(E_1,\,E_2,\,D\right)\right)\right)
\leqslant \int\limits_{A} Q(x)\cdot \eta^p(|x-x_0|)\ dm(x)
\end{equation}
для каждой измеримой функции $\eta : (r_1,r_2)\rightarrow [0,\infty
],$ такой что
\begin{equation}\label{eq28*}
\int\limits_{r_1}^{r_2}\eta(r)\ dr\ \ge\ 1\,.
\end{equation}
Соотношению (\ref{eq2*!}), в частности, все открытые дискретные
отображения $f\in W_{loc}^{1, n}(D),$ мера множества точек ветвления
которых равна нулю и внутренняя дилатация $K_I(x, f)$ которых
локально интегрируема (см. \cite[теорема~1 и следствие~2]{Sev$_1$}).
Следующее утверждение для случая гомеоморфизмов на плоскости
доказано в \cite[лемма~5.1]{GRY}. В нашем случае речь идёт о
ситуации пространства ${\Bbb R}^n,$ $n\geqslant 2,$ для отображений
со значительно более общими топологическими свойствами.

\medskip
\begin{lemma}\label{lem1}
{\,\sl Пусть $n\geqslant 2,$ $p\geqslant 1,$ область $D\subset {\Bbb
R}^n$ регулярна, а $D^{\,\prime}\subset {\Bbb R}^n$ ограничена и
имеет локально квазиконформную границу, являющуюся сильно достижимой
относительно $p$-модуля. Пусть также отображение $f:D\rightarrow
D^{\,\prime},$ $D^{\,\prime}=f(D),$ является кольцевым
$Q$-отображением относительно $p$-модуля во всех точках
$x_0\in\partial D,$ кроме того, $f$ является открытым, дискретным и
замкнутым. Тогда $f$ продолжается до непрерывного отображения
$f:\overline{D}_P\rightarrow \overline{D^{\,\prime}}_P,$
$f(\overline{D}_P)=\overline{D^{\,\prime}}_P,$ если найдётся
измеримая по Лебегу функция $\psi:(0, \infty)\rightarrow [0,
\infty]$ такая, что
\begin{equation} \label{eq5B}
I(\varepsilon,
\varepsilon_0):=\int\limits_{\varepsilon}^{\varepsilon_0}\psi(t)dt <
\infty
\end{equation}
при всех $\varepsilon\in (0,\varepsilon_0)$ и, кроме того,
$I(\varepsilon, \varepsilon_0)\rightarrow\infty$ при
$\varepsilon\rightarrow 0,$ и при $\varepsilon\rightarrow 0$
\begin{equation} \label{eq4*}
\int\limits_{\varepsilon<|x-x_0|<\varepsilon_0}Q(x)\cdot\psi^p(|x-x_0|)
\ dm(x)\,=\,o\left(I^p(\varepsilon, \varepsilon_0)\right)\,.
\end{equation}}
\end{lemma}

\begin{proof}
Так как $D^{\,\prime}$ имеет локально квазиконформную границу, то
$\overline{D^{\,\prime}}_P=\overline{D^{\,\prime}}$ (см.
\cite[теорема~4.1]{Na}). В силу метризуемости пространства
$\overline{D}_P$ достаточно доказать, что для каждого простого конца
$P$ области $D$ предельное множество
$$L=C(f, P):=\left\{y\in{{\Bbb
R}^n}:y=\lim\limits_{k\rightarrow\infty}f(x_k),x_k\rightarrow
P,x_k\in D\right\}$$ состоит из единственной точки $y_0\in\partial
D^{\,\prime}$. Заметим, что $L\ne\varnothing$ в силу компактности
множества $\overline{D^{\,\prime}}$, и $L$ является подмножеством
$\partial D^{\,\prime}$ ввиду замкнутости отображения $f$ (см.
\cite[теорема~3.3]{Vu$_1$}). Предположим, что существуют, по крайней
мере, две точки $y_0$ и $z_0\in L$. То есть, найдётся не менее двух
последовательностей $x_k, x_k^{\,\prime}\in D,$ таких, что
$x_k\rightarrow P$ и $x^{\,\prime}_k\rightarrow P$ при
$k\rightarrow\infty,$ и при этом, $f(x_k)\rightarrow y_0$ и
$f(x^{\,\prime}_k)\rightarrow z_0$ при $k\rightarrow\infty.$ В силу
\cite[лемма~2]{KR} каждый простой конец $P$ регулярной области $D$ в
${\Bbb R}^n$, $n\geqslant 2,$ содержит цепь разрезов $\sigma_k$,
лежащую на сферах $S_k$ с центром в некоторой точке $x_0\in\partial
D$ и с евклидовыми радиусами $r_k\rightarrow 0$ при
$k\rightarrow\infty$. Пусть $D_k$ -- области, ассоциированные с
разрезами $\sigma_k$, $k=1,2,\ldots$. Не ограничивая общности
рассуждений, переходя к подпоследовательности, если это необходимо,
мы можем считать, что $x_k, x_k^{\,\prime}\in D_k.$ В самом деле,
так как последовательности $x_k$ и $x_k^{\,\prime}$ сходятся к
простому концу $P,$ найдётся номер $k_1\in {\Bbb N}$ такой, что
$x_{k_1}, x_{k_1}^{\,\prime}\in D_1.$ Далее, найдётся номер $k_2\in
{\Bbb N},$ $k_2>k_1,$ такой, что $x_{k_2}, x_{k_2}^{\,\prime}\in
D_2.$ И так далее. Вообще, на $m$-м шаге мы найдём номер $k_m\in
{\Bbb N},$ $k_m>k_{m-1},$ такой, что $x_{k_m}, x_{k_m}^{\,\prime}\in
D_m.$ Продолжая этот процесс, мы получим две последовательности
$x_{k_m}$ и $x^{\,\prime}_{k_m},$ принадлежащие области $D_m,$
сходящиеся к $P$ при $m\rightarrow\infty$ и такие, что
$f(x_{k_m})\rightarrow y_0$ и $f(x^{\,\prime}_{k_m})\rightarrow y_0$
при $m\rightarrow\infty.$ Переобозначая, если это необходимо,
$x_{k_m}\mapsto x_m,$ мы получаем последовательность $x_m$ с
требуемыми свойствами.

Отметим, что $y_0$ и $z_0\in
\partial D^{\,\prime},$ поскольку по условию $C(f,
\partial D)\subset \partial D^{\,\prime}.$ По определению сильно достижимой границы в
точке $y_0\in \partial D^{\,\prime},$ для любой окрестности $U$ этой
точки найдутся компакт $C_0^{\,\prime}\subset D^{\,\prime},$
окрестность $V$ точки $y_0,$ $V\subset U,$ и число $\delta>0$ такие,
что
\begin{equation}\label{eq1B}
M_p(\Gamma(C_0^{\,\prime}, F, D^{\,\prime}))\ge \delta
>0
\end{equation} для произвольного континуума
$F,$ пересекающего $\partial U$ и $\partial V.$ В силу предположения
$C(f,
\partial D)\subset \partial D^{\,\prime},$ имеем, что для
$C_0:=f^{\,-1}(C_0^{\,\prime})$ выполнено условие $C_0\cap \partial
D=\varnothing$ (см. \cite[теорема~3.3]{Vu$_1$}). Поскольку
$I(P)=\bigcap\limits_{m=1}\limits^{\infty}\overline{D_m}\subset
\partial D$ (см. \cite[предложение~1]{KR}), то не ограничивая общности рассуждений, можно считать, что
$C_0\cap\overline{D_k}=\varnothing$ для каждого $k\in {\Bbb N.}$
Соединим точки $x_k$ и $x_k^{\,\prime}$ кривой $\gamma_k,$ лежащей в
$D_k.$ Заметим, что $f(x_k)\in V$ и $f(x_k^{\,\prime})\in D\setminus
\overline{U}$ при всех достаточно больших $k\in {\Bbb N}.$ В таком
случае, найдётся номер $k_0\in {\Bbb N},$ такой, что согласно
(\ref{eq1B})
\begin{equation}\label{eq2}
M_p(\Gamma(C_0^{\,\prime}, |f(\gamma_k)|, D^{\,\prime}))\ge \delta
>0
\end{equation}
при всех $k\ge k_0\in {\Bbb N}.$ Обозначим через $\Gamma_k$
семейство всех полуоткрытых кривых $\beta:[a, b)\rightarrow {\Bbb
R}^n$ таких, что $\beta(a)\in |f(\gamma_k)|,$ $\beta(t)\in
D^{\,\prime}$ при всех $t\in [a, b)$ и, кроме того,
$\lim\limits_{t\rightarrow b-0}\beta(t):=B\in C_0^{\,\prime}.$
Очевидно, что
\begin{equation}\label{eq4E}
M_p(\Gamma_k)=M_p\left(\Gamma\left(C_0^{\,\prime}, |f(\gamma_k)|,
D^{\,\prime}\right)\right)\,.
\end{equation}
При каждом фиксированном $k\in {\Bbb N},$ $k\ge k_0,$ рассмотрим
семейство $\Gamma_k^{\,\prime}$ (полных) поднятий $\alpha:[a,
b)\rightarrow D$ семейства $\Gamma_k$ с началом во множестве
$|\gamma_k|,$ т.е., $f\circ\alpha=\beta,$ $\beta\in\Gamma_k$ и
$\alpha(a)\in |\gamma_k|.$ Поскольку $f$ -- замкнутое отображение,
то оно сохраняет границу (см. \cite[теорема~3.3]{Vu$_1$}) и, значит,
кривая $\alpha$ лежит в $D$ вместе со своим замыканием. Более того,
в силу дискретности отображения $f$ кривая $\alpha$ может быть
продолжена до непрерывной кривой $\overline{\alpha}:[a,
b]\rightarrow D.$ Заметим, что по определению
$\overline{\beta}(b)\in C_0^{\,\prime},$ так что
$\overline{\alpha}(1)\in C_0$ по определению множества $C_0.$
Значит, $\overline{\alpha}\in \Gamma(|\gamma_k|, C_0, D),$ где
$\overline{\alpha}$ обозначает продолженную кривую
$\overline{\alpha}:[a, b]\rightarrow D.$ Погрузим компакт $C_0$ в
некоторый континуум $C_1,$ всё ещё полностью лежащий в области $D$
(см.~\cite[лемма~1]{Smol}). Можно снова считать, что
$C_1\cap\overline{D_k}=\varnothing,$ $k=1,2,\ldots.$ Заметим, что
$\Gamma(|\gamma_k|, C_0, D)>\Gamma(\sigma_k, C_1, D),$ при этом,
$|\gamma_k|$ и $C_0$ -- континуумы в $D,$ а $\sigma_k$ -- разрез
соответствующий области $D_k.$ Поэтому к семейству кривых
$\Gamma(\sigma_k, C_1, D)$ можно применить определение кольцевого
$Q$-отображения (\ref{eq2*!}). В связи с этим, заметим, что
$\sigma_k\in S(x_0, r_k)$ для некоторой точки $x_0\in \partial D$ и
некоторой последовательности $r_k>0,$ $r_k\rightarrow 0$ при
$k\rightarrow\infty$ (см. \cite[лемма~2]{KR}). Здесь, не ограничивая
общности рассуждений, можно считать, что ${\rm dist}\,(x_0,
C_1)>\varepsilon_0.$ Кроме того, заметим, что функция
$$\eta_k(t)=\left\{
\begin{array}{rr}
\psi(t)/I(r_k, \varepsilon_0), &   t\in (r_k,
\varepsilon_0),\\
0,  &  t\in {\Bbb R}\setminus (r_k, \varepsilon_0)\,,
\end{array}
\right. $$
где $I(\varepsilon,
\varepsilon_0):=\int\limits_{\varepsilon}^{\varepsilon_0}\psi(t)dt,$
удовлетворяет условию нормировки вида (\ref{eq28*}). По доказанному
$\Gamma_k^{\,\prime}\subset\Gamma(|\gamma_k|, C_0, D),$ так что
$M_p(f(\Gamma_k^{\,\prime}))\leqslant M_p(f(\Gamma(|\gamma_k|, C_0,
D))).$ Поэтому, в силу определения кольцевого $Q$-отоб\-ра\-же\-ния
в граничной точке относительно $p$-модуля, а также ввиду условий
(\ref{eq5B})--(\ref{eq4*}),
\begin{equation}\label{eq11*}
M_p(f(\Gamma_k^{\,\prime}))\leqslant M_p(f(\Gamma(|\gamma_k|, C_0,
D)))\leqslant M_p(f(\Gamma(\sigma_k, C_1, D))\leqslant \Delta(k)\,,
\end{equation}
где $\Delta(k)\rightarrow 0$ при $k\rightarrow \infty.$ Однако,
$\Gamma_k=f(\Gamma_k^{\,\prime}),$ поэтому из (\ref{eq11*}) получим,
что при $k\rightarrow \infty$
\begin{equation}\label{eq3E}
M_p(\Gamma_k)= M_p\left(f(\Gamma_k^{\,\prime})\right)\leqslant
\Delta(k)\rightarrow 0\,.
\end{equation}
Однако, соотношение (\ref{eq3E}) вместе с равенством (\ref{eq4E})
противоречат неравенству (\ref{eq2}), что и доказывает лемму.~$\Box$
\end{proof}

\medskip
{\it Доказательство теоремы \ref{th4}}. В первом случае полагаем
$\psi(r)=\frac{1}{r^{\frac{n-1}{p-1}}q_{x_0}^{\frac{1}{p-1}}(r)}$
при $r\in (0, \varepsilon_0)$ и $\psi(r)=0$ при $r>\varepsilon_0.$ В
обозначениях леммы \ref{lem1} $I:=I(\varepsilon,
\varepsilon_0)=\int\limits_{\varepsilon}^{\varepsilon_0}\
\frac{dr}{r^{\frac{n-1}{p-1}}q_{x_0}^{\frac{1}{p-1}}(r)},$ где
$q_{x_0}(r):=\frac{1}{\omega_{n-1}r^{n-1}}\int\limits_{|x-x_0|=r}Q(x)\,d{\mathcal
H}^{n-1},$ Тогда
\begin{equation}\label{eq10B}
\frac{\omega_{n-1}}{I^{p-1}}=\frac{1}{I^p}\int\limits_{\varepsilon<|x-x_0|<\varepsilon_0}
Q(x)\cdot \psi^p(|x-x_0|)\ dm(x)\rightarrow 0, \quad
\varepsilon\rightarrow 0\,.
\end{equation}
Из (\ref{eq10B}) с учётом (\ref{eq10A}) вытекает, что оба условия
(\ref{eq5B})--(\ref{eq4*}) выполнены и, таким образом, случай 1)
вытекает непосредственно из леммы \ref{lem1}.

Рассмотрим случай 2). Заметим, что для функций класса $FMO$ в точке
$x_0$ выполнено условие (\ref{eq31*}) при  $\varepsilon \rightarrow
0 $ и для некоторого $e_0>0,$ $e_0 \leqslant {\rm
dist}\,\left(0,\partial D\right).$ При $\varepsilon_0<e_0$ полагаем
$\psi(t):=\frac{1}{\left(t\,\log{\frac1t}\right)^{n/{p}}},$
$I(\varepsilon,
\varepsilon_0):=\int\limits_{\varepsilon}^{\varepsilon_0}\psi(t) dt
\geqslant \log{\frac{\log{\frac{1}
{\varepsilon}}}{\log{\frac{1}{\varepsilon_0}}}}$ и
$\eta(t):=\psi(t)/I(\varepsilon, \varepsilon_0).$ Заметим, что
$\int\limits_{\varepsilon}^{\varepsilon_0}\eta(t)dt=1,$ кроме того,
из соотношения (\ref{eq31*}) вытекает, что
\begin{equation}\label{eq32*}
\frac{1}{I^{p}(\varepsilon,
\varepsilon_0)}\int\limits_{\varepsilon<|x|<\varepsilon_0}
Q(x+x_0)\cdot\psi^{p}(|x|)
 \ dm(x)\leqslant C\left(\log\log\frac{1}{\varepsilon}\right)^{1-{p}}\rightarrow
 0
 \end{equation}
при $\varepsilon\rightarrow 0.$ Итак, из (\ref{eq32*}) вытекает
выполнение условий (\ref{eq5B})--(\ref{eq4*}) леммы \ref{lem1},
откуда и следует случай 2). Теорема доказана.~$\Box$

\medskip
{\bf 5. О равностепенной непрерывности отображений в замыкании
области с сильно достижимой границей.} Наша ближайшая цель --
установить свойство равностепенной непрерывности семейств
отображений, удовлетворяющих оценкам вида (\ref{eq1A}), а также
классов Орлича--Соболева $W_{loc}^{1, \varphi}.$ Речь идёт о
равностепенной непрерывности не только во внутренних точках области,
но и в её замыкании. Всюду далее, если не оговорено противное, $d$
-- одна из метрик в пространстве простых концов в заданной области
$D,$ упомянутых перед формулировкой теоремы \ref{th3}, а
$d^{\,\prime}$ -- евклидова метрика.

\medskip
Перед тем, как переходить к формулировке и доказательству основных
утверждений настоящего раздела, сформулируем следующее утверждение,
доказательство которого аналогично случаю гомеоморфизмов при $p=n$
(см. \cite[следствие~5]{KRSS}).

\medskip
\begin{theorem}\label{th4A}
{\,\sl Пусть $x_0\in \partial D,$ ограниченное отображение
$f:D\rightarrow {\Bbb R}^n$ является нижним $Q$-гомеоморфизмом
относительно $p$-модуля в области $D\subset{\Bbb R}^n,$ $Q\in
L_{loc}^{\frac{n-1}{p-n+1}}({\Bbb R}^n),$ $p>n-1$ и
$\alpha:=\frac{p}{p-n+1}.$ Тогда $f$ является кольцевым
$Q^{\frac{n-1}{p-n+1}}$-гомеоморфизмом в этой же точке.
 }
\end{theorem}

\medskip
\begin{proof} Зафиксируем $\varepsilon_0\in(0,d_0),$ $d_0=\sup\limits_{x\in
D}|x-x_0|.$ Пусть $\varepsilon\in(0, \varepsilon_0)$ и пусть
континуумы $C_1$ и $C_2$ удовлетворяют условиям  $C_1\subset
\overline{B(x_0, \varepsilon)\cap D}$ и $C_2\subset D\setminus
B(x_0, \varepsilon_0).$ Рассмотрим семейство множеств
$\Gamma_{\varepsilon}:=\bigcup\limits_{r\in (\varepsilon,
\varepsilon_0)}\{f(S(x_0, r)\cap D)\}.$ Заметим, что множество
$\sigma_r:=f(S(x_0, r)\cap D)$ замкнуто в $f(D)$ как гомеоморфный
образ замкнутого множества $S(x_0, r)\cap D$ в $D.$ Кроме того,
заметим, что $\sigma_r$ при $r\in (\varepsilon, \varepsilon_0)$
отделяет $f(C_1)$ от $f(C_2)$ в $f(D),$ поскольку
$$f(C_1)\subset f(B(x_0, r)\cap D):=A,\quad
f(C_2)\subset f(D)\setminus \overline{f(B(x_0, r)\cap D)}:=B\,,$$
$A$ и $B$ открыты в $f(D)$ и
$$f(D)=A\cup \sigma_r\cup B\,.$$

\medskip
Пусть $\Sigma_{\varepsilon}$ -- семейство всех множеств, отделяющих
$f(C_1)$ от $f(C_2)$ в $f(D).$ Пусть $\rho^{n-1}\in \widetilde{{\rm
adm}}\bigcup\limits_{r\in (\varepsilon, \varepsilon_0)} f(S(x_0,
r)\cap D)$ в смысле соотношения (\ref{eq13.4.13}), тогда также
$\rho\in {\rm adm}\bigcup\limits_{r\in (\varepsilon, \varepsilon_0)}
f(S(x_0, r)\cap D)$ в смысле соотношения (\ref{eq8.2.6}) при
$k=n-1.$ Следовательно, так как
$\widetilde{M}_{q}(\Sigma_{\varepsilon})\geqslant
M_{q(n-1)}(\Sigma_{\varepsilon})$ при произвольном $q\geqslant 1,$
то
$$\widetilde{M}_{p/(n-1)}(\Sigma_{\varepsilon})\geqslant$$
\begin{equation}\label{eq5AA}
\geqslant \widetilde{M}_{p/(n-1)}\left(\bigcup \limits_{r\in
(\varepsilon, \varepsilon_0)} f(S(x_0, r)\cap D)\right)\geqslant
M_p\left(\bigcup\limits_{r\in (\varepsilon, \varepsilon_0)} f(S(x_0,
r)\cap D)\right)\,.
\end{equation}
Однако, ввиду (\ref{eq3}) и (\ref{eq4}),
$$
\widetilde{M}_{p/(n-1)}(\Sigma_{\varepsilon})=\frac{1}{(M_{\alpha}(\Gamma(f(C_1),
f(C_2), f(D))))^{1/(\alpha-1)}}\,, \alpha=p/(p-n+1)\,.
$$
По лемме \ref{lem4}
$$M_p\left(\bigcup\limits_{r\in (\varepsilon, \varepsilon_0)} f(S(x_0,
r)\cap D)\right)\geqslant
$$
\begin{equation}\label{eq8B}
\geqslant \int\limits_{\varepsilon}^{\varepsilon_0}
\frac{dr}{\Vert\,Q\Vert_{s}(r)}=
\int\limits_{\varepsilon}^{\varepsilon_0}
\frac{dt}{\omega^{\frac{p-n+1}{n-1}}_{n-1}
t^{\frac{n-1}{\alpha-1}}\widetilde{q}_{x_0}^{\,\frac{1}{\alpha-1}}(t)}\quad\forall\,\,
i\in {\Bbb N}\,, s=\frac{n-1}{p-n+1}\,,\end{equation} где
$\Vert
Q\Vert_{s}(r)=\left(\int\limits_{D(x_0,r)}Q^{s}(x)\,d{\mathcal{A}}\right)^{\frac{1}{s}}$
-- $L_{s}$-норма функции $Q$ над сферой $S(x_0,r)\cap D,$ а
$\widetilde{q}_{x_0}(r)$ -- её среднее значение над этой сферой.
Тогда из (\ref{eq5AA})--(\ref{eq8B}) вытекает, что
\begin{equation}\label{eq9C}
M_{\alpha}(\Gamma(f(C_1), f(C_2), f(D)))\leqslant
\frac{\omega_{n-1}}{I^{\alpha-1}}\,,
\end{equation}
где $I=\int\limits_{\varepsilon}^{\varepsilon_0}\
\frac{dr}{r^{\frac{n-1}{\alpha-1}}\widetilde{q}_{x_0}^{\frac{1}{\alpha-1}}(r)}.$
Заметим, что $f(\Gamma(C_1,C_2, D))=\Gamma(f(C_1), f(C_2), f(D)),$
так что из (\ref{eq9C}) вытекает, что
$$
M_{\alpha}(f(\Gamma(C_1,C_2, D)))\leqslant
\frac{\omega_{n-1}}{I^{\alpha-1}}\,.
$$
Завершает доказательство применение предложения \ref{pr1A}.~$\Box$
\end{proof}

Ниже мы ограничимся ситуацией, когда все рассматриваемые отображения
являются гомеоморфизмами. Для $n-1<p,$ областей $D,$
$D^{\,\prime}\subset {\Bbb R}^n,$ $z_1, z_2\in D,$ $z_1\ne z_2,$
$z_1^{\prime},$ $z_2^{\prime}\in D^{\prime}$ и произвольной
измеримой по Лебегу функции $Q$ обозначим символом $\frak{R}_{Q,
p}^{z_1, z_2, z_1^{\,\prime}, z_2^{\,\prime}}(D, D^{\,\prime})$
семейство всех нижних кольцевых $Q$-гомеоморфизмов $f:D\rightarrow
D^{\,\prime}$ относительно $p$-модуля в $\overline{D},$
$f(D)=D^{\,\prime},$ таких что
\begin{equation}\label{eq8***} f(z_1)=z_1^{\prime},\quad
f(z_2)=z_2^{\prime}\,.\end{equation}

\medskip
Имеет место следующее утверждение.

\medskip
\begin{lemma}\label{th5}{\,\sl
Пусть $n\geqslant 2,$ область $D\subset {\Bbb R}^n$ регулярна, а
область $D^{\,\prime}\subset {\Bbb R}^n$ имеет локально
квазиконформную границу, которая является сильно достижимой
относительно $\alpha$-модуля, $\alpha:=\frac{p}{p-n+1}.$
Предположим, $Q\in L_{loc}^{\frac{n-1}{p-n+1}}({\Bbb R}^n),$ и что
для каждого $x_0\in \overline{D}$ найдется $\delta_0=\delta(x_0)>0$,
такое, что при всех $\varepsilon\in (0, \delta_0)$ и некоторой
измеримой по Лебегу функции $\psi:(0, \infty)\rightarrow [0,
\infty]$
\begin{equation} \label{eq5D}
0<I(\varepsilon,
\delta_0):=\int\limits_{\varepsilon}^{\delta_0}\psi(t)dt <
\infty\,,\,\,\, I(\varepsilon, \delta_0)\rightarrow
\infty\quad\text{при}\quad\varepsilon\rightarrow 0\,,
\end{equation}
и, кроме того, при $\varepsilon\rightarrow 0$
\begin{equation} \label{eq4B}
\int\limits_{\varepsilon<|x-x_0|<\delta_0}Q^{s}(x)\cdot\psi^{\,\alpha}(|x-x_0|)
\ dm(x)\,=\,o\left(I^{\,\alpha}(\varepsilon, \delta_0)\right)\,,
\end{equation}
$ \alpha:=\frac{p}{p-n+1}, s=\frac{n-1}{p-n+1}.$

Тогда каждый элемент $f\in \frak{R}_{Q, p}^{z_1, z_2,
z_1^{\,\prime}, z_2^{\,\prime}}(D, D^{\,\prime})$ продолжается до
непрерывного отображения $\overline f\colon\overline
D_P\rightarrow\overline{D^{\,\prime}}_P$, при этом, $\frak{R}_{Q,
p}^{z_1, z_2, z_1^{\,\prime}, z_2^{\,\prime}}(\overline{D}_P,
\overline{D^{\,\prime}}_P),$ состоящее из всех продолженных таким
образом отображений, является равностепенно непрерывным, а значит, и
нормальным  в $\overline{D}_P$.}
\end{lemma}

\begin{proof} Во внутренних точках области $D$ семейство 
$\frak{R}_{Q, p}^{z_1, z_2, z_1^{\,\prime}, z_2^{\,\prime}}(D,
D^{\,\prime})$ является равностепенно непрерывным по следующим
соображениям: каждое отображение $f\in\frak{R}_{Q, p}^{z_1, z_2,
z_1^{\,\prime}, z_2^{\,\prime}}(D, D^{\,\prime})$ является так
называемым $Q^{\frac{n-1}{p-n+1}}$-гомеоморфизмом относительно
$\alpha$-модуля, $\alpha:=\frac{p}{p-n+1}$ (см.
\cite[теорема~13.1]{RSS} при $p=n$ и \cite[теорема~7.1]{GS}), а
семейства таких отображений (фиксирующих две и более точек при
$p=n$, и произвольные семейства при $n-1<p<n$) равностепенно
непрерывны при условиях (\ref{eq5D})--(\ref{eq4B}) на функцию $Q$
(см. \cite[лемма~6.1]{RSS} и \cite[лемма 2.4]{GSS}). Возможность
непрерывного граничного продолжения каждого элемента
$f\in\frak{R}_{Q, p}^{z_1, z_2, z_1^{\,\prime}, z_2^{\,\prime}}(D,
D^{\,\prime})$ вытекает из \cite[теорема~2 и соотношение
(37)]{Sev$_3$}.

Покажем равностепенную непрерывность семейства $\frak{R}_{Q,
p}^{z_1, z_2, z_1^{\,\prime}, z_2^{\,\prime}}(\overline{D}_P,
\overline{D^{\,\prime}}_P)$ в точках $E_D,$ где $E_D$ --
пространство простых концов в области $D.$  Не ограничивая общности,
можно считать, что $\overline{D^{\prime}}_P=\overline{D^{\prime}}.$
Предположим противное, а именно, что $\frak{R}_{Q, p}^{z_1, z_2,
z_1^{\,\prime}, z_2^{\,\prime}}(\overline{D}_P,
\overline{D^{\,\prime}}_P)$ не является равностепенно непрерывным в
некоторой точке $P_0\in E_D.$
Тогда найдутся число $a>0,$ последовательность $P_k\in
\overline{D}_P,$ $k=1,2,\ldots$ и элементы $f_k\in\frak{R}_{Q,
p}^{z_1, z_2, z_1^{\,\prime}, z_2^{\,\prime}}(\overline{D}_P,
\overline{D^{\,\prime}}_P)$ такие, что $d(P_k, P_0)<1/k$ и
\begin{equation}\label{eq3C}
|f_k(P_k)-f_k(P_0)|\geqslant a\quad\forall\quad k=1,2,\ldots,\,.
\end{equation}
Ввиду возможности непрерывного продолжения каждого $f_k$ на границу
$D$ в терминах простых концов, для всякого $k\in {\Bbb N}$ найдётся
элемент $x_k\in D$ такой, что $d(x_k, P_k)<1/k$ и
$|f_k(x_k)-f_k(P_k)|<1/k.$ Тогда из (\ref{eq3C}) вытекает, что
\begin{equation}\label{eq4C}
|f_k(x_k)-f_k(P_0)|\geqslant a/2\quad\forall\quad k=1,2,\ldots,\,.
\end{equation}
Аналогично, в силу непрерывного продолжения отображения $f_k$ в
$\overline{D}_P$ найдётся последовательность $x_k^{\,\prime}\in D,$
$x_k^{\,\prime}\rightarrow P_0$ при $k\rightarrow \infty$ такая, что
$|f_k(x_k^{\,\prime})-f_k(P_0)|<1/k$ при $k=1,2,\ldots\,.$ Тогда из
(\ref{eq4C}) вытекает, что
$$
|f_k(x_k)-f_k(x_k^{\,\prime})|\geqslant a/4\quad\forall\quad
k=1,2,\ldots\,,
$$
где последовательности $x_k$ и $x_k^{\,\prime}$ принадлежат $D$ и
сходятся к простому концу $P_0$ при $k\rightarrow\infty.$ В силу
компактности множества $\overline{D^{\,\prime}}$ последовательность
$f_k(P_0)$ имеет сходящуюся подпоследовательность $f_{k_l}(P_0),$
которая сходится к некоторой точке $y_0\in \partial D^{\,\prime}$
при $l\rightarrow \infty.$ Поскольку
$|f_k(x_k^{\,\prime})-f_k(P_0)|<1/k$ при $k=1,2,\ldots\,,$ мы
получим, что $f_{k_l}(x_{k_l}^{\,\prime})\rightarrow y_0$ при
$l\rightarrow \infty.$ Не ограничивая общности рассуждений, можно
считать, что сама последовательность $f_k(x_k^{\,\prime})$ сходится
при $k\rightarrow \infty$ к $y_0.$

Положим $U=B(y_0,r_0)$, где $0<r_0<a/4$. В силу \cite[лемма~2]{KR}
простой конец $P_0$ регулярной области $D$ в ${\Bbb R}^n$,
$n\geqslant 2,$ содержит цепь разрезов $\sigma_k$, лежащую на сферах
$S_k$ с центром в некоторой точке $x_0\in\partial D$ и с евклидовыми
радиусами $r_k\rightarrow 0$ при $k\rightarrow\infty$. Пусть $D_k$
-- области, ассоциированные с разрезами $\sigma_k$, $k=1,2,\ldots$.
Поскольку последовательности $x_k$ и $x_k^{\,\prime}$ сходятся к
простому концу $P_0$ при $k\rightarrow\infty,$ мы можем считать, что
точки $y_k=f_k(x_k)$ и $y^{\,\prime}_k=f_k(x^{\,\prime}_k)$
принадлежат области $D_{k}^{\,\prime}=f(D_{k}).$ Соединим точки
$y_k$ и $y_k^{\,\prime}$ кривой $C_k,$ полностью лежащей в
$D_k^{\,\prime}.$ Заметим, что по построению $\partial U\cap
|C_k|\ne\varnothing$ (как обычно, $|C_k|$ обозначает носитель кривой
$C_k$).

Поскольку область $D^{\,\prime}$ сильно достижима относительно
$\alpha$-модуля, существует континуум $E\subset D^{\,\prime}$ и
число $\delta>0$, такие, что
\begin{equation}\label{eq1E}
M_{\alpha}(\Gamma(E, |C_k|, D^{\,\prime}))\geqslant\delta
\end{equation}
для всех достаточно больших $k$.

Без ограничения общности можем считать, что последнее условие
выполнено для всех $k=1,2,\ldots$. Заметим, что поскольку семейство
отображений $\frak{R}_{Q, p}^{z_1, z_2, z_1^{\,\prime},
z_2^{\,\prime}}(D, D^{\,\prime})$ равностепенно непрерывно в области
$D,$ а $\overline{D^{\,\prime}}$ является компактом, то
$\frak{R}_{Q, p}^{z_1, z_2, z_1^{\,\prime}, z_2^{\,\prime}}(D,
D^{\,\prime})$ нормально в этой области (см.
\cite[теорема~20.4]{Va}). Следовательно, не ограничивая общности,
можно считать, что последовательность $f_k$ сходится локально
равномерно к некоторому непрерывному отображению $f,$ более того,
предельное отображение $f$ является гомеоморфизмом или постоянной в
$D$ (см. \cite[теорема~1]{Cr} при $p\ne n$ и \cite[лемма~4.2]{RSS}
при $p=n$). Так как $f_m$ удовлетворяют условиям нормировки
(\ref{eq8***}), отображение $f$ является гомеоморфизмом. Тогда также
$f_k^{\,-1}\rightarrow f^{\,-1}$ при $k\rightarrow\infty$ (см.
\cite[лемма~3.1]{RSS}). Ввиду включения $E\subset f(D)$, компакты
$K_k:=f_k^{\,-1}(E)$ при $k\rightarrow\infty$ сходятся к компакту
$f^{\,-1}(E)$ в смысле хаусдорфовой метрики. Тогда при всех
$k\geqslant k_0$ все множества $K_k$ принадлежат некоторой
$\varepsilon$-окрестности компакта $f^{\,-1}(E),$ замыкание которой
мы обозначим через $K_0.$ Можно считать, что $K_0$ -- компакт в $D.$
Учитывая, что $x_0\in\partial D$, имеем
\begin{equation*}
\varepsilon_2:={\rm dist\,}(x_0, K_0)>0.
\end{equation*}
Полагаем $\varepsilon_0:=\min\{\delta_0, \varepsilon_2\}$.
Рассмотрим семейство кривых $\Gamma_k,$ соединяющих множества $K_k$
и $|f_k^{\,-1}(C_k)|$ в области $D.$ Заметим, что
$|f_k^{\,-1}(C_k)|\subset D_k,$ $K_k\subset D\setminus B(x_0,
\varepsilon_0)$ и $K_0\subset D\setminus B(x_0, \varepsilon_0),$ так
что $\Gamma_k>\Gamma(S(x_0, r_k), K_k, D).$ Тогда при всех
$k\geqslant k_0$
$$M_{\alpha}(f_k(\Gamma_k))\leqslant M_{\alpha}(f_k(\Gamma(S(x_0, r_k), K_k,
D)))\leqslant
$$
\begin{equation}\label{eq8BA}
\leqslant M_{\alpha}(f_k(\Gamma(S(x_0, r_k), K_0, D)))\leqslant
\end{equation}
$$\leqslant
\int\limits_{A(x_0, r_k, \varepsilon_0)}Q^{\frac{n-1}{p-n+1}}(x)
\eta^{\,\alpha}(|x-x_0|)\,dm(x)\,, \alpha=\frac{p}{p-n+1}\,,$$
где $A(x_0, \varepsilon, \varepsilon_0)=\{x\in {\Bbb R}^n:
\varepsilon<|x-x_0|<\varepsilon_0\}$ и $\eta: (r_k,
\varepsilon_0)\rightarrow [0,\infty]$ -- произвольная измеримая по
Лебегу функция такая, что
\begin{equation*}
\int\limits_{r_k}^{\varepsilon_0}\eta(r)dr=1\,.
\end{equation*}
Так как $I(\varepsilon, \varepsilon_0)\rightarrow\infty$ при
$\varepsilon\rightarrow 0,$ то при достаточно малых $r_k>0$ имеем:
$I(r_k, \varepsilon_0)>0.$ Полагаем
\begin{equation*}
\eta(t)=\begin{cases} \psi(t)/I(r_k, \varepsilon_0), &
t\in(r_k, \varepsilon_0),\\
0, & t\notin(r_k, \varepsilon_0)\,.
\end{cases}
\end{equation*}
В таком случае,
\begin{equation*}
\int\limits_{r_k}^{\varepsilon_0}\eta(t)\,dt=\frac{1}{I(r_k,
\varepsilon_0)}\int\limits_{r_k}^{\varepsilon_0}\psi(t)\,dt=1.
\end{equation*}
Тогда из (\ref{eq8BA}) с учётом (\ref{eq4B}) следует, что при
$k\rightarrow\infty$
\begin{equation}\label{eq9A}
M_{\alpha}(f_k(\Gamma_k))\rightarrow 0\,.
\end{equation}
Однако, $f_k(\Gamma_k)=\Gamma(E, |C_k|, D^{\,\prime}),$ откуда из
(\ref{eq9A}) мы получаем, что
$$
M_{\alpha}(\Gamma(E, |C_k|, D^{\,\prime}))\rightarrow
0\quad\text{при}\quad k\rightarrow\infty\,,
$$
что противоречит соотношению (\ref{eq1E}). Полученное противоречие
указывает на то, что изначальное предположение об отсутствии
равностепенной непрерывности рассматриваемого семейства было
неверным.~$\Box$
\end{proof}

\medskip
Из леммы \ref{th5} мы получаем следующее утверждение.

\medskip
\begin{theorem}\label{th6}{\,\sl
Пусть $n\geqslant 2,$ область $D\subset {\Bbb R}^n$ регулярна, а
область $D^{\,\prime}\subset {\Bbb R}^n$ имеет локально
квазиконформную границу, которая является сильно достижимой
относительно $\alpha$-модуля, $\alpha:=\frac{p}{p-n+1}.$
Предположим, $Q\in L_{loc}^{\frac{n-1}{p-n+1}}({\Bbb R}^n)$ и
выполнено одно из следующих условий:

1) либо в каждой точке $x_0\in \overline{D}$ при некотором
$\varepsilon_0=\varepsilon_0(x_0)>0$ и всех
$0<\varepsilon<\varepsilon_0$
$$
\int\limits_{\varepsilon}^{\varepsilon_0}
\frac{dt}{t^{\frac{n-1}{\alpha-1}}\widetilde{q}_{x_0}^{\,\frac{1}{\alpha-1}}(t)}<\infty\,,\qquad
\int\limits_{0}^{\varepsilon_0}
\frac{dt}{t^{\frac{n-1}{\alpha-1}}\widetilde{q}_{x_0}^{\,\frac{1}{\alpha-1}}(t)}=\infty\,,
$$
где
$\widetilde{q}_{x_0}(r):=\frac{1}{\omega_{n-1}r^{n-1}}\int\limits_{|x-x_0|=r}Q^{\frac{n-1}{p-n+1}}(x)\,d{\mathcal
H}^{n-1};$

2) либо $Q^{\frac{n-1}{p-n+1}}\in FMO(\overline{D}).$

Тогда каждый элемент $f\in \frak{R}_{Q, p}^{z_1, z_2,
z_1^{\,\prime}, z_2^{\,\prime}}(D, D^{\,\prime})$ продолжается до
непрерывного отображения $\overline f\colon\overline
D_P\rightarrow\overline {D^{\,\prime}_P}$, при этом, семейство
отображений $\frak{R}_{Q, p}^{z_1, z_2, z_1^{\,\prime},
z_2^{\,\prime}}(\overline{D}_P, \overline{D^{\,\prime}}_P),$
состоящее из всех продолженных таким образом отображений, является
равностепенно непрерывным, а значит, и нормальным  в
$\overline{D}_P$.}
\end{theorem}

\medskip
{\it Доказательство} вытекает из леммы \ref{th5} по аналогии с
доказательством теоремы \ref{th4}.~$\Box$

\medskip
{\it Доказательство теоремы \ref{th7}.} По лемме \ref{thOS4.1}
каждое $f\in \frak{F}_{ \varphi, Q, \alpha}^{z_1, z_2,
z_1^{\,\prime}, z_2^{\,\prime}}(D, D^{\,\prime})$ является нижним
$B$-отображением относительно $p$-модуля при $B(x)=
Q^{\frac{p-n+1}{n-1}}(x, f),$ где $p$ находится из условия
$\alpha=\frac{p}{p-n+1}.$ Однако, относительно $B(x)$ выполнены
условия 1) и 2) теоремы \ref{th6}, поскольку
$B^{\frac{n-1}{p-n+1}}(x)=Q(x),$ где $Q$ удовлетворяет соотношениям
1)-2) теоремы \ref{th7}. Оставшаяся часть утверждения вытекает из
теоремы \ref{th6}.~$\Box$

\medskip
{\bf 6. О равностепенной непрерывности нижних $Q$-отображений в
$QED$-об\-лас\-тях.} В предыдущем параграфе речь шла об
отображениях, фиксирующих два внутренних значения $z_1, z_2\in D.$
Значительно важнее, однако, было бы получить результаты в том
случае, когда фиксируется только одна внутренняя точка заданной
области. Сказанное иллюстрируется простым фактом из теории
конформных отображений: например, хорошо известно, что существует и
единственно отображение единичного круга на себя, фиксирующее одну
внутреннюю и одну граничную точку единичного круга; однако, о
существовании отображений, фиксирующих две и более внутренние точки,
вообще говоря, ничего нельзя сказать (см. \cite[пункт 10, $\S\, 3,$
гл.~I]{Shab}). В дальнейшем мы для простоты ограничимся случаем
$p=n,$ где $p$ -- порядок модуля семейств кривых.

\medskip
Напомним, что {\it сферическое (хордальное)} расстояние между
точками $x$ и $y$  в $\overline{{\Bbb R}^n}={\Bbb
R}^n\cup\{\infty\}$ есть величина
$$
h(x,\infty)=\frac{1}{\sqrt{1+{|x|}^2}}\,,\,\,
h(x,y)=\frac{|x-y|}{\sqrt{1+{|x|}^2} \sqrt{1+{|y|}^2}},\ \ x\ne
\infty\ne y\,.
$$
{\it Хордальным диаметром} множества $E\,\subset\,\overline{{\Bbb
R}^n}$ называется величина
\begin{equation}\label{eq47***}
h(E)\,=\,\sup\limits_{x\,,y\,\in\,E}\,h(x,y)\,.
\end{equation}

\begin{proposition}\label{pr9}{\,\sl Для любых непересекающихся невырожденных
континуумов $B$ и $F$ в $\overline{{\Bbb R}^n}$ и некоторой
постоянной $\lambda_n>0$ имеет место соотношение:
$$M\left(\Gamma\left(F, B, \overline{{\Bbb R}^n} \right)\right)\ \geqslant\
\frac {\omega_{n-1}}{\,\,\,\,{\left[\log
 \frac{2\lambda_n^2}{h (F)h
(B)}\right]}^{n-1}}\,,$$
см. \cite[(7.29), разд. 7.4, гл. 7]{MRSY}.
 }\end{proposition}

\medskip
Имеет место следующее утверждение, обобщающее
\cite[лемма~3.1]{Sev$_3$} в случае не локально связных границ.

\medskip
\begin{lemma}\label{lem3}
{\,\sl Пусть область $D$ регулярна, область $D^{\,\prime}$
ограничена, имеет локально квазиконформную границу и, одновременно,
является $QED$-областью. Пусть также $P_0$ -- некоторый простой
конец в $E_D,$ а $\sigma_m,$ $m=1,2,\ldots,$ -- соответствующая ему
цепь разрезов, лежащих на сферах с центром в некоторой точке $x_0\in
\partial D$ и радиусов $r_m\rightarrow 0,$ $m\rightarrow\infty.$
Пусть $D_m$ -- соответствующая $P_0$ последовательность
ассоциированных областей, а $C_m$ -- произвольная последовательность
континуумов, принадлежащих $D_m.$

Предположим, $Q\in L_{loc}^{n-1}({\Bbb R}^n),$ $f:D\rightarrow
D^{\,\prime}$ -- нижний $Q$-гомеоморфизм, $f(D)=D^{\,\prime},$ такой
что $b_0^{\,\prime}=f(b_0)$ для некоторых $b_0\in D$ и
$b_0^{\,\prime}\in D^{\,\prime}.$ Пусть также найдётся
$\varepsilon_0=\varepsilon(x_0)>0,$
такое, что при некотором $0<p<n$ выполнено условие
\begin{equation}\label{eq5***}
\int\limits_{A(x_0, \varepsilon, \varepsilon_0)}
Q^{n-1}(x)\cdot\psi^{\,n}(|x-x_0|)\,dm(x) \leqslant K\cdot
I^p(\varepsilon, \varepsilon_0)\,,
\end{equation}
где сферическое кольцо $A(x_0, \varepsilon, \varepsilon_0)$
определено как в (\ref{eq49***}), а $\psi$ -- некоторая
неотрицательная измеримая функция, такая, что при всех
$\varepsilon\in(0, \varepsilon_0)$
\begin{equation}\label{eq7***}
I(\varepsilon,
\varepsilon_0):=\int\limits_{\varepsilon}^{\varepsilon_0}\psi(t)\,dt
< \infty\,,
\end{equation}
при этом, $I(\varepsilon, \varepsilon_0)\rightarrow \infty$ при
$\varepsilon\rightarrow 0.$

Тогда найдётся число
$\widetilde{\varepsilon_0}=\widetilde{\varepsilon_0}(x_0)\in (0,
\varepsilon_0)$  и $M_0\in {\Bbb N}$ такие, что
\begin{equation}
\label{eq3.10} h(f(C_m))\leqslant \frac{\alpha_n}{\delta}\,
\exp\left\{-\beta_n I^{\gamma_{n,p}}(r_m, \varepsilon_0)\cdot
(\alpha(r_m))^{-1/(n-1)}\right\}\qquad \forall\,m\geqslant M_0\,,
\end{equation}
где величина $h(f(C_m))$ в левой части (\ref{eq3.10}) определена в
(\ref{eq47***}),
\begin{equation}\label{eq1.3}
\alpha(\sigma)=\left(
1+\frac{\int\limits_{\widetilde{\varepsilon_0}}^{\varepsilon_0}\psi(t)\,dt}
{\int\limits_{\sigma}^{\widetilde{\varepsilon_0}}\psi(t)\,dt}\right)^n\,,\end{equation}
$\delta = \frac{1}{2}\cdot h\left(b_0^{\,\prime},
\partial D^{\,\prime}\right),$ $h$ -- хордальное расстояние
между множествами,
$\alpha_n$ ~--- некоторая постоянная, зависящая только от $n,$ $A$
~--- постоянная, участвующая в определении $QED$-об\-лас\-ти
$D^{\,\prime},$ см. (\ref{eq4***}), $\beta_n =
\left(\frac{\omega_{n-1}}{KA}\right)^{\frac{1}{n-1}}$ и степень
$\gamma_{n,p}=1-\frac{p-1}{n-1}.$ }
\end{lemma}

\medskip
\begin{proof} Прежде всего, в сделанных выше обозначениях положим $\eta_{\varepsilon}(t)=\psi(t)/I(\varepsilon, \varepsilon_0).$
Тогда при каждом $\varepsilon\in(0, \varepsilon_0)$ согласно
предложению \ref{pr1A} (см. также \cite[лемма~7.4]{MRSY})
\begin{equation}\label{eq2A}
\frac{\omega_{n-1}}{J^{n-1}}\leqslant \frac{1}{I^n(\varepsilon,
\varepsilon_0)}\int\limits_{A(x_0, \varepsilon, \varepsilon_0)}
Q^{n-1}(x)\cdot \psi^n(|x-x_0|) dm(x)\,,
\end{equation}
где $\omega_{n-1}$ -- площадь единичной сферы в ${\Bbb R}^n,$
$$J =J(\varepsilon,
\varepsilon_0):=\int\limits_{\varepsilon}^{\varepsilon_0}\
\frac{dr}{r\widetilde{q}_{x_0}^{\frac{1}{n-1}}(r)}\,,$$
\begin{equation}\label{eq5G}
\widetilde{q}_{x_0}(r)=\frac{1}{\omega_{n-1}r^{n-1}}\int\limits_{|x-x_0|=r}Q^{n-1}(x)\,d{\mathcal
H}^{n-1}\,,
\end{equation}
$A(x_0, \varepsilon, \varepsilon_0)=\{x\in {\Bbb R}^n:
\varepsilon<|x-x_0|<\varepsilon_0\}.$ В силу условия (\ref{eq5***})
и того, что $I(\varepsilon, \varepsilon_0)\rightarrow \infty$ при
$\varepsilon\rightarrow 0,$ правая часть соотношения (\ref{eq2A})
стремится к нулю при $\varepsilon\rightarrow 0.$ В таком случае, из
(\ref{eq2A}) следует, что $J(\varepsilon, \varepsilon_0)\rightarrow
\infty$ при $\varepsilon\rightarrow 0.$ Тогда ввиду
\cite[теорема~3]{KR} отображение $f$ продолжается до гомеоморфизма
$\overline{D}_P$ на $\overline{D^{\,\prime}}_P.$ В частности,
поскольку область $D^{\,\prime}$ ограничена, отсюда следует, что
пространство $\overline{D}_P$ содержит не менее двух простых концов
$P_1$ и $P_2\in E_D,$ где $E_D$ -- пространство простых концов в
области $D.$

Пусть теперь $P_1\subset E_D$ -- простой конец, не совпадающий с
$P_0,$ где $P_0$ -- фиксированный простой конец из условия леммы.
Предположим, $G_m,$ $m=1,2,\ldots,$ -- последовательность областей,
соответствующая простому концу $P_1$ и $x_m\in D$ -- произвольная
последовательность точек, такая что $x_m\rightarrow P_1$ при
$m\rightarrow\infty.$ Можно считать, что $x_m\in G_m$ для всякого
$m\in {\Bbb N}.$ Тогда, так как $f$ имеет непрерывное продолжение на
$\overline{D}_P,$ то $f(x_m)\rightarrow f(P_1)$ при
$m\rightarrow\infty.$ Из последнего соотношения вытекает, что
$h(f(x_m), f(P_1))\rightarrow 0$ при $m\rightarrow\infty.$ (Здесь,
как и выше, мы отождествляем $f(P_1)$ с соответствующей точкой
границы области $D^{\,\prime},$ а хордальное расстояние $h(f(x_m),
f(P_1))$ понимается как хордальное расстояние между $f(x_m)$ и этой
точкой в $\overline{{\Bbb R}^n}$). Заметим, что при всех $m\geqslant
m_0$ и некотором $m_0\in {\Bbb N}$
$$h(f(b_0), f(x_m))=h(b_0^{\,\prime}, f(x_m))\geqslant
h(b_0^{\,\prime},f(P_1))-h(f(x_m), f(P_1))\geqslant$$
\begin{equation}\label{eq3D}
 \geqslant\frac{1}{2}\cdot
h(b_0^{\,\prime}, \partial D^{\,\prime}):=\delta\,,
\end{equation}
где $h(b_0^{\,\prime}, \partial D^{\,\prime})$ обозначает хордальное
расстояние между $b_0^{\,\prime}$ и $\partial D^{\,\prime}.$
Построим последовательность континуумов $K_m,$ $m=1,2,\ldots,$
следующим образом. Соединим точку $x_1$ с точкой $b_0$ произвольной
кривой в $D,$ которую мы обозначим через $K_1.$ Далее, соединим
точки $x_2$ и $x_1$ кривой $K_1^{\prime},$ лежащей в $G_1.$
Объединив кривые $K_1$ и $K_1^{\prime},$ получим кривую $K_2,$
соединяющую точки $b_0$ и $x_2.$ И так далее. Пусть на некотором
шаге мы имеем кривую $K_m,$ соединяющую точки $x_m$ и $b_0.$
Соединим точки $x_{m+1}$ и $x_m$ кривой $K_m^{\,\prime},$ лежащей в
$G_m.$ Объединяя между собой кривые $K_m$ и $K_m^{\,\prime},$
получим кривую $K_{m+1}.$ И так далее.

Пусть $C_m,$ $m=1,2,\ldots,$ -- последовательность континуумов в
областях $D_m,$ заданная по условию. Покажем, что найдётся номер
$m_1\in {\Bbb N},$ такой что
\begin{equation}\label{eq4D}
D_m\cap K_m=\varnothing\quad\forall\quad m\geqslant m_1\,.
\end{equation}
Предположим, что (\ref{eq4D}) не имеет места, тогда найдутся
возрастающая последовательность номеров $m_k\rightarrow\infty,$
$k\rightarrow\infty,$ и последовательность точек $\xi_k\in
K_{m_k}\cap D_{m_k},$ $m=1,2,\ldots,\,.$ Тогда, с одной стороны,
$\xi_k \rightarrow P_0$ при $k\rightarrow\infty.$

\medskip
Рассмотрим следующую процедуру. Заметим, что возможны два случая:
либо все элементы $\xi_k$ при $k=1,2,\ldots$ принадлежат множеству
$D\setminus G_1,$ либо найдётся номер $k_1$ такой, что $\xi_{k_1}\in
G_1.$ Далее, рассмотрим последовательность $\xi_k,$ $k>k_1.$
Заметим, что возможны два случая: либо $\xi_k$ при $k>k_1$
принадлежат множеству $D\setminus G_2,$ либо найдётся номер
$k_2>k_1$ такой, что $\xi_{k_2}\in G_2.$ И так далее. Предположим,
элемент $\xi_{k_{l-1}}\in G_{l-1}$ построен. Заметим, что возможны
два случая: либо $\xi_k$ при $k>k_{l-1}$ принадлежат множеству
$D\setminus G_l,$ либо найдётся номер $k_l>k_{l-1}$ такой, что
$\xi_{k_l}\in G_l.$ И так далее. Эта процедура может быть как
конечной (оборваться на каком-то $l\in {\Bbb N}$), так и
бесконечной, в зависимости от чего мы имеем две ситуации:

1) либо найдутся номера $n_0\in {\Bbb N}$ и $l_0\in {\Bbb N}$ такие,
что $\xi_k\in D\setminus G_{n_0}$ при всех $k>l_0;$

2) либо для каждого $l\in {\Bbb N}$ найдётся элемент $\xi_{k_l}$
такой, что $\xi_{k_l}\in G_l,$ причём последовательность $k_l$
является возрастающей по $l\in {\Bbb N}.$

\medskip
Рассмотрим каждый из этих случаев и покажем, что в обоих из них мы
приходим к противоречию. Пусть имеет место ситуация 1), тогда
заметим, что все элементы последовательности $\xi_k$ принадлежат
$K_{n_0},$ откуда вытекает существование подпоследовательности
$\xi_{k_r},$ $r=1,2,\ldots,$ сходящейся при $r\rightarrow\infty$ к
некоторой точке $\xi_0\in D.$ Однако, с другой стороны $\xi_k\in
D_{m_k}$ и, значит, $\xi_0\in \bigcap\limits_{m=1}^{\infty}
\overline{D_m}\subset
\partial D$ (см. \cite[предложение~1]{KR} по этому поводу).
Полученное противоречие говорит о том, что случай 1) невозможен.
Пусть имеет место случай 2), тогда одновременно $\xi_k\rightarrow
P_0$ и $\xi_k\rightarrow P_1$ при $k\rightarrow\infty.$ В силу
непрерывного продолжения $f$ на $\overline{D}_P$ отсюда имеем, что
$f(\xi_k)\rightarrow f(P_0)$ и $f(\xi_k)\rightarrow f(P_1)$ при
$k\rightarrow\infty,$ откуда $f(P_0)=f(P_1),$ что противоречит
гомеоморфности продолжения $f$ в $\overline{D}_P.$ Полученное
противоречие указывает на справедливость соотношения (\ref{eq4D}).

Положим теперь $\widetilde{\varepsilon_0}=\min\{\varepsilon_0,
r_{m_1+1}\},$ и пусть $M_0$ -- натуральное число, такое что
$r_m<\widetilde{\varepsilon_0}$ при всех $m\geqslant M_0.$
Рассмотрим измеримую функцию
$$\eta_{m}(t)= \left\{
\begin{array}{rr}
\psi(t)/I(r_m, \widetilde{\varepsilon_0}), &   t\in (r_m, \widetilde{\varepsilon_0}),\\
0,  &  t\not\in (r_m, \widetilde{\varepsilon_0})\,,
\end{array}
\right.$$
%
%
где, как и прежде, величина $I(a, b)$ определяется соотношением
$I(a, b)=\int\limits_a^b\psi(t)\,dt.$ Заметим, что функция
$\eta_m(t)$ удовлетворяет соотношению вида (\ref{eq28*}), где вместо
$r_1$ и $r_2$ участвуют $r_m$ и $\widetilde{\varepsilon_0},$
соответственно. Заметим, что ввиду соотношения (\ref{eq4D}), а также
по определению разрезов $\sigma_m\subset r_m,$ $\Gamma\left(C_m,
K_m, D\right)>\Gamma(S(x_0, r_m), S(x_0, \widetilde{\varepsilon_0}),
D)$ и значит, $f(\Gamma\left(C_m, K_m, D\right))>f(\Gamma(S(x_0,
r_m), S(x_0, \widetilde{\varepsilon_0}), D )),$ откуда
$$M(f(\Gamma\left(C_m, K_m, D\right)))\leqslant M(f(\Gamma(S(x_0,
r_m), S(x_0, \widetilde{\varepsilon_0}), D))$$
(см. \cite[теорема~6.4]{Va}). В таком случае, согласно теореме
\ref{th4A}, мы получим, что
$$M\left(\Gamma\left(f(C_m), f(K_m),
D^{\,\prime}\right)\right)=$$
\begin{equation}\label{eq37***}
=M\left(f\left(\Gamma\left(C_m, K_m, D\right)\right)\right)\leqslant
M(f(\Gamma(S(x_0, r_m), S(x_0, \widetilde{\varepsilon_0}), D
)))\leqslant
\end{equation}
$$\leqslant\,\frac{K\cdot I^{p}(r_m, \varepsilon_0)}{I^n(r_m,
\widetilde{\varepsilon_0})}=K\cdot I^{p-n}(r_m,
\varepsilon_0)\cdot\alpha(r_m)\,,\quad m\geqslant M_0\,,$$
где $\alpha(r_m)$ определяется из соотношения (\ref{eq1.3}) при
$\sigma=r_m.$
Т.к. по условию область $D^{\,\prime}=f(D)$ является
$QED$-об\-лас\-тью, то при некоторой постоянной $A<\infty,$ см.
(\ref{eq4***}), из (\ref{eq37***}) получим
\begin{equation}\label{eq38***}
M\left(\Gamma\left(f(C_m), f(K_m), \overline{{\Bbb
R}^n}\right)\right)\leqslant\,K\cdot A\cdot I^{p-n}(r_m,
\varepsilon_0)\cdot\alpha(r_m)\,.
\end{equation}
Тогда, по предложению \ref{pr9}, из (\ref{eq38***}) получаем, что
$$
\frac {\omega_{n-1}}{\,\,\,\,{\left[\log
 \frac{2\lambda_n^2}{h (f(C_m))h
(f(K_m))}\right]}^{n-1}}\leqslant\,K\cdot A\cdot I^{p-n}(r_m,
\varepsilon_0)\cdot\alpha(r_m)\,,\quad m\geqslant M_0\,,
$$
откуда
\begin{equation}\label{eq40***}
h(f(C_m))\leqslant \frac{2\lambda_n^2}{h(f(K_m))}\exp\left\{-
\left(\frac{\omega_{n-1}}{KA}\right)^{\frac{1}{n-1}}I^{\frac{n-p}{n-1}}
(r_m, \varepsilon_0)\cdot(\alpha(r_m))^{-\frac{1}{n-1}}\right\}\,,
\end{equation}
$m\geqslant M_0.$ Заметим, что ввиду (\ref{eq3D}), из
(\ref{eq40***}) следует, что
\begin{equation*}
h(f(C_m))\leqslant \frac{2\lambda_n^2}{\delta}\exp\left\{-
\left(\frac{\omega_{n-1}}{KA}\right)^{\frac{1}{n-1}}I^{\frac{n-p}{n-1}}
(r_m, \varepsilon_0)\cdot(\alpha(r_m))^{-\frac{1}{n-1}}\right\}\,,
m\geqslant M_0\,.
\end{equation*}
Лемма доказана.~$\Box$
\end{proof}

\medskip
Для заданных областей $D,$ $D^{\,\prime}\subset {\Bbb R}^n,$ $n\ge
2,$ измеримой по Лебегу функции $Q,$ $b_0\in D,$ $b_0^{\,\prime}\in
D^{\,\prime},$ обозначим через $\frak{G}_{b_0, b_0^{\,\prime},
Q}\left(D, D^{\,\prime}\right)$ семейство всех нижних
$Q$-гомеоморфизмов $f:D\rightarrow D^{\,\prime},$ таких что
$f(D)=D^{\,\prime},$ $b_0^{\,\prime}=f(b_0).$ В наиболее общей
ситуации основное утверждение настоящего раздела может быть
сформулировано следующим образом.

\medskip
\begin{lemma}\label{lem3A}
{\,\sl Пусть область $D$ регулярна, область $D^{\,\prime}$
ограничена, имеет локально квазиконформную границу и, одновременно,
является $QED$-областью.

Предположим, что $Q\in L_{loc}^{n-1}({\Bbb R}^n),$ и что для каждого
$x_0\in \overline{D}$ найдётся $\varepsilon_0=\varepsilon(x_0)>0,$
такое, что при некотором $0<p<n$ выполнено условие (\ref{eq5***}),
где сферическое кольцо $A(x_0, \varepsilon, \varepsilon_0)$
определено как в (\ref{eq49***}), а $\psi$ -- некоторая
неотрицательная измеримая функция, такая, что при всех
$\varepsilon\in(0, \varepsilon_0)$ выполнено условие (\ref{eq7***}),
при этом, $I(\varepsilon, \varepsilon_0)\rightarrow \infty$ при
$\varepsilon\rightarrow 0.$

Тогда каждое $f\in\frak{G}_{b_0, b_0^{\,\prime}, Q}\left(D,
D^{\,\prime}\right)$ продолжается до гомеоморфизма
$f:\overline{D}_P\rightarrow \overline{D^{\,\prime}}_P,$ при этом
семейство таким образом продолженных отображений является
равностепенно непрерывным в $\overline{D}_P.$ }
\end{lemma}

\begin{proof} В силу условия
(\ref{eq5***}) и того, что $I(\varepsilon, \varepsilon_0)\rightarrow
\infty$ при $\varepsilon\rightarrow 0,$ правая часть соотношения
(\ref{eq2A}) стремится к нулю при $\varepsilon\rightarrow 0.$ В
таком случае, из (\ref{eq2A}) следует, что $J(\varepsilon,
\varepsilon_0)\rightarrow \infty$ при $\varepsilon\rightarrow 0,$
где $J$ определено в (\ref{eq5G}). Тогда ввиду \cite[теорема~3]{KR}
отображение $f\in\frak{G}_{b_0, b_0^{\,\prime}, Q}\left(D,
D^{\,\prime}\right)$ продолжается до гомеоморфизма $\overline{D}_P$
на $\overline{D^{\,\prime}}_P.$ Равностепенная непрерывность
семейства $\frak{G}_{b_0, b_0^{\,\prime}, Q}\left(D,
D^{\,\prime}\right)$ во внутренних точках области $D$ следует,
например, из комбинации теоремы \ref{th4A} и \cite[лемма~7.6]{MRSY}.
Осталось показать равностепенную непрерывность  $\frak{G}_{b_0,
b_0^{\,\prime}, Q}\left(\overline{D}_P,
\overline{D^{\,\prime}}_P\right)$ на $E_D.$

Предположим противное, а именно, что семейство $\frak{G}_{b_0,
b_0^{\,\prime}, Q}\left(\overline{D}_P,
\overline{D^{\,\prime}}_P\right)$ не является равностепенно
непрерывным в некоторой точке $P_0\in E_D.$ Тогда найдутся число
$a>0,$ последовательность $P_k\in \overline{D}_P,$ $k=1,2,\ldots$ и
элементы $f_k\in\frak{G}_{b_0, b_0^{\,\prime},
Q}\left(\overline{D}_P, \overline{D^{\,\prime}}_P\right)$ такие, что
$d(P_k, P_0)<1/k$ и
\begin{equation}\label{eq6E}
|f_k(P_k)-f_k(P_0)|\geqslant a\quad\forall\quad k=1,2,\ldots,\,.
\end{equation}
Ввиду возможности непрерывного продолжения каждого $f_k$ на границу
$D$ в терминах простых концов, для всякого $k\in {\Bbb N}$ найдётся
элемент $x_k\in D$ такой, что $d(x_k, P_k)<1/k$ и
$|f_k(x_k)-f_k(P_k)|<1/k.$ Тогда из (\ref{eq6E}) вытекает, что
\begin{equation}\label{eq7E}
|f_k(x_k)-f_k(P_0)|\geqslant a/2\quad\forall\quad k=1,2,\ldots,\,.
\end{equation}
Аналогично, в силу непрерывного продолжения отображения $f_k$ в
$\overline{D_P}$ найдётся последовательность $x_k^{\,\prime}\in D,$
$x_k^{\,\prime}\rightarrow P_0$ при $k\rightarrow \infty$ такая, что
$|f_k(x_k^{\,\prime})-f_k(P_0)|<1/k$ при $k=1,2,\ldots\,.$ Тогда из
(\ref{eq7E}) вытекает, что
\begin{equation}\label{eq8C}
|f_k(x_k)-f_k(x_k^{\,\prime})|\geqslant a/4\quad\forall\quad
k=1,2,\ldots\,,\,.
\end{equation}
Пусть $\sigma_m,$ $m=1,2,\ldots,$ -- соответствующая $P_0$ цепь
разрезов, лежащих на сферах с центром в некоторой точке $x_0\in
\partial D$ и радиусов $r_m\rightarrow 0,$ $m\rightarrow\infty.$
Пусть $D_m$ -- соответствующая $P_0$ последовательность
ассоциированных областей. Не ограничивая общности рассуждений, можно
считать, что $x_k$ и $x_k^{\,\prime}$ принадлежат области $D_k.$
Соединим точки $x_k$ и $x_k^{\,\prime}$ кривой $C_k$ лежащей в
$D_k.$ Тогда по лемме \ref{lem3} мы получим, что
$h(|f(C_k)|)\rightarrow 0$ при $k\rightarrow \infty,$ что
противоречит неравенству (\ref{eq8C}). Полученное противоречие
указывает на то, что исходное предположение об отсутствии
равностепенной непрерывности семейства $\frak{G}_{b_0,
b_0^{\,\prime}, Q}\left(\overline{D}_P,
\overline{D^{\,\prime}}_P\right)$ было неверным.~$\Box$
\end{proof}

\medskip
Из леммы \ref{lem3A} на основе рассуждений, приведённых при
доказательстве теоремы \ref{th4}, получаем следующее утверждение.

\medskip
\begin{theorem}\label{th8}{\,\sl
Пусть область $D$ регулярна, область $D^{\,\prime}$ ограничена,
имеет локально квазиконформную границу и, одновременно, является
$QED$-областью.

Предположим, что $Q\in L_{loc}^{n-1}({\Bbb R}^n),$ и что для каждого
$x_0\in \overline{D}$ выполнено одно из следующих условий:

1) либо $Q^{n-1}\in FMO(\overline{D});$

2) либо в каждой точке $x_0\in \overline{D}$ при некотором
$\varepsilon_0=\varepsilon_0(x_0)>0$ и всех
$0<\varepsilon<\varepsilon_0$
$$
\int\limits_{\varepsilon}^{\varepsilon_0}
\frac{dt}{t\widetilde{q}_{x_0}^{\,\frac{1}{n-1}}(t)}<\infty\,,\qquad
\int\limits_{0}^{\varepsilon_0}
\frac{dt}{t\widetilde{q}_{x_0}^{\,\frac{1}{n-1}}(t)}=\infty\,,
$$
где
$\widetilde{q}_{x_0}(r):=\frac{1}{\omega_{n-1}r^{n-1}}\int\limits_{|x-x_0|=r}Q^{n-1}(x)\,d{\mathcal
H}^{n-1}.$

Тогда каждое $f\in\frak{G}_{b_0, b_0^{\,\prime}, Q}\left(D,
D^{\,\prime}\right)$ продолжается до гомеоморфизма
$f:\overline{D}_P\rightarrow \overline{D^{\,\prime}}_P,$ при этом
семейство таким образом продолженных отображений является
равностепенно непрерывным в $\overline{D}_P.$}
\end{theorem}

\medskip
{\it Доказательство теоремы \ref{th7A}}. Утверждение теоремы
вытекает из теоремы \ref{th8}. Действительно, согласно лемме
\ref{thOS4.1} каждое $f\in \frak{F}_{b_0, b_0^{\,\prime}, \varphi,
Q}(D, D^{\,\prime})$ является нижним кольцевым
$Q^{1/(n-1)}$-отображением в $\overline{D}.$ В таком случае,
желанное заключение прямо вытекает из теоремы \ref{th8}.~$\Box$

\medskip
КОНТАКТНАЯ ИНФОРМАЦИЯ

\medskip
\noindent{{\bf Евгений Александрович Севостьянов} \\
Житомирский государственный университет им.\ И.~Франко\\
ул. Большая Бердичевская, 40 \\
г.~Житомир, Украина, 10 008 \\ тел. +38 066 959 50 34 (моб.),
e-mail: esevostyanov2009@mail.ru}

\end{document}